\documentclass{article}

\usepackage{amscd}
\usepackage{tikz}
\usepackage{tikz-cd}
\usepackage{graphicx}
\usepackage{multirow}
\usepackage{amsmath,amssymb,amsfonts}
\usepackage{amsthm}
\usepackage{mathrsfs}
\usepackage[title]{appendix}
\usepackage{xcolor}
\usepackage{textcomp}
\usepackage{manyfoot}
\usepackage{booktabs}
\usepackage{algorithm}
\usepackage{algorithmicx}
\usepackage{algpseudocode}
\usepackage{listings}
\usepackage{hyperref}
\usepackage[top=3.5cm, bottom=4cm, left=3.5cm, right=3.5cm]{geometry}
\usepackage{tocbibind}

\newtheorem{theorem}{Theorem}
\newtheorem{proposition}[theorem]{Proposition}% 

\newtheorem*{example}{Example}%
\newtheorem{remark}[theorem]{Remark}%
\newtheorem*{main theorem}{Main Theorem}
\newtheorem{definition}[theorem]{Definition}%

\title{Orbifold Mapping Spaces via Bibundles}
\author{Yoshihiro Sugimoto}
\date{}

\begin{document}

\maketitle

\begin{abstract}
The existence of orbifold structures on mapping spaces between orbifolds was established in previous work, notably by W. Chen \cite{C} and B. Chen-Du-Liao \cite{CDL}. In this paper, we construct an infinite-dimensional orbifold structure on the mapping space in a more geometric way, using bibundles (or Hilsum--Skandalis morphisms), based on Lerman's description of orbifold morphisms \cite{L}. We construct a proper \'{e}tale Banach (or Fr\'{e}chet) Lie groupoid representing the mapping space in the ${C^k}$, $C^{\infty}$, and ${W^{k,p}}$ settings, and show that its Morita equivalence class is independent of the choices of groupoid representations. Locally, the mapping groupoid is isomorphic to the action groupoid of a finite group, where the finite group is identified with the automorphism group of the corresponding bibundle. This provides a concrete geometric framework for further studies of mapping spaces and moduli spaces on orbifolds, with applications to Floer theory on symplectic orbifolds.
\end{abstract}

\tableofcontents

\section{Introduction}
The concept of an orbifold was introduced by I. Satake under the name ``V-manifold'' \cite{S1,S2}. His original motivation was to extend the classical theory of manifolds to spaces that locally look like the quotient of a Euclidean space by a finite group action. In the 1970s, W. Thurston rediscovered these objects in the context of $3$-manifold topology and introduced the term ``orbifold'' \cite{T}. Since then, the theory of orbifolds has developed in a more categorical direction. In modern geometry, orbifolds are recognized as a special class of differentiable stacks, which can be represented by proper \'{e}tale Lie groupoids up to Morita equivalence. In the study of orbifolds, the space of maps ${\mathrm{Map}(X,Y)}$ between orbifolds $X$ and $Y$ serves as a fundamental object of study. To study these spaces rigorously, one needs to endow them with a suitable infinite-dimensional orbifold structure.

The existence of such a structure has been established in previous work. The foundational work of W. Chen \cite{C} gave a direct construction of the orbifold structure on the mapping space using local charts. Subsequently, B. Chen, Du and Liao \cite{CDL} developed a more categorical treatment within the framework of groupoids and their functors. Both approaches involve refinements and equivalences of local data. In the groupoid formulation, these are expressed in terms of localization and $2$-categorical data. These approaches provide natural descriptions of the mapping space, while a more direct geometric description is useful for applications in geometric analysis.

The objective of this paper is to provide a more concrete and geometric construction of mapping spaces between orbifolds using a geometric description of orbifold morphisms. Following Lerman \cite{L}, we represent orbifold morphisms using bibundles (also known as Hilsum--Skandalis morphisms). We use this description to construct an infinite-dimensional orbifold structure on the mapping space. We can now state the main result of this paper informally as follows. The precise statement is given in Section 4.
\begin{main theorem}[Theorem 31]
Let $X$ and $Y$ be orbifolds with compact coarse moduli spaces. Then, the mapping space ${\mathfrak{Map}(X,Y)}$ admits a natural structure of an infinite-dimensional orbifold. The construction is based on a geometric description of orbifold morphisms by bibundles and is independent of the choices of groupoid representations of $X$ and $Y$. Moreover, the coarse moduli space of this infinite-dimensional orbifold is naturally identified with the space of orbifold morphisms from $X$ to $Y$. Our construction also gives a concrete description of the local orbifold structure: locally, the mapping groupoid is isomorphic to the action groupoid of a finite group, where the finite group is the automorphism group of the corresponding bibundle.
\end{main theorem}

This work is motivated by further developments in Floer theory on symplectic orbifolds. The mapping spaces constructed here provide a geometric framework for studying moduli spaces arising in Floer theory on orbifolds. In particular, this framework is expected to facilitate the construction of Kuranishi structures for relevant moduli spaces. In a forthcoming sequel, we will apply this framework to develop Morse homology and Hamiltonian Floer homology on orbifolds.

The remainder of this paper is organized as follows. In Section 2, we review the background on groupoids, Morita equivalences, and orbifold groupoids. In Section 3, we discuss the localization of groupoids with respect to equivalences. We review Lerman's work \cite{L}, which shows that morphisms in the localized category are represented by equivalence classes of bibundles between Lie groupoids (also known as Hilsum--Skandalis morphisms). Section 4 is the core of this paper. We construct the mapping groupoid and establish the infinite-dimensional orbifold structure on the mapping space. We then show that its coarse moduli space is naturally isomorphic to the space of orbifold morphisms. Thus, the coarse moduli space agrees with the $1$-categorical mapping space. The precise form of the main theorem is given at the end of Section 4.

\section*{Acknowledgements}
The author is supported by JSPS KAKENHI Grant Number JP26K16985.

\section*{AI disclosure}
Large language models were used to assist with English language editing. The author takes full responsibility for the mathematical content and results of the manuscript.

\section{Lie groupoids and orbifolds}

In this section, we review the basic definitions and fix notation for Lie groupoids, largely following \cite{ALR}.

\subsection{Foundations of topological and Lie groupoids}

A \emph{groupoid} is a category in which every morphism is an isomorphism. A groupoid $\mathcal{G}$ consists of a set ${G_0}$ of objects and a set ${G_1}$ of morphisms, equipped with several structural maps: the source and target maps ${s,t\colon G_1\rightarrow G_0}$, the unit map ${u\colon G_0\rightarrow G_1}$, the inverse map ${i\colon G_1\rightarrow G_1}$, and the multiplication map ${G_1\times_{s,t}G_1\rightarrow G_1}$. These maps are required to satisfy the standard axioms of a category and the invertibility of morphisms.

The central objects of study in this paper are groupoids with geometric data. In the following sections, we will enrich these sets with topological or differentiable structures by assuming that $G_0$ and $G_1$ are topological spaces or smooth manifolds, and that all the structural maps are continuous or smooth, respectively.

\begin{definition}
A groupoid ${\mathcal{G}=(G_0,G_1)}$ is said to be a \emph{topological groupoid} if $G_0$ and $G_1$ are topological spaces and the structural maps satisfy the following conditions:
\begin{enumerate}
    \item The source map ${s\colon G_1\rightarrow G_0}$, which assigns each morphism ${g\in G_1}$ its source ${s(g)\in G_0}$, is continuous.
    \item The target map ${t\colon G_1\rightarrow G_0}$, which assigns each morphism ${g\in G_1}$ its target ${t(g)\in G_0}$, is continuous.
    \item The composition map ${m\colon G_1\times_{s,t}G_1\rightarrow G_1}$, which assigns ${g\in G_1}$ and ${h\in G_1}$ with ${s(g)=t(h)}$ their composition ${gh\in G_1}$, is continuous.
    \item The unit (or identity) map ${u\colon G_0\rightarrow G_1}$, which assigns each ${x\in G_0}$ the identity ${u(x)\colon x\rightarrow x}$, is continuous. Note that ${u(x)}$ is a two-sided unit for the composition.
    \item The inverse map ${i\colon G_1\rightarrow G_1}$, which assigns each ${g\in G_1}$ its inverse ${i(g)=g^{-1}}$, is continuous.
\end{enumerate}
\end{definition}

For a pair of objects ${x,y\in G_0}$, the set of morphisms from $x$ to $y$ is denoted by ${G_1(x,y)}$. 

\begin{gather*}
    G_1(x,y)=\{g\in G_1 \ | \ s(g)=x, t(g)=y\}
\end{gather*}

\begin{definition}
    A topological groupoid $\mathcal{G}$ is called a \emph{Lie groupoid} if the following conditions hold:
    \begin{enumerate}
        \item $G_0$ and $G_1$ are smooth manifolds.
        \item All structural maps $s$, $t$, $u$, $i$ and $m$ are smooth.
        \item The source map ${s\colon G_1\rightarrow G_0}$ is a smooth submersion.
    \end{enumerate}
    Note that condition ${(3)}$, together with the smoothness of the inverse map $i$ in ${(2)}$, implies that the target map ${t=s\circ i}$ is also a smooth submersion. This ensures that the fiber product ${G_1\times_{s,t}G_1}$ is a smooth manifold, making the smoothness of $m$ in ${(2)}$ well-defined.
\end{definition}

For a groupoid $\mathcal{G}$, the following relation defined in ${G_0\times G_0}$

\begin{gather*}
    \sim\stackrel{\mathrm{def}}{=}\{(x,y)\in G_0\times G_0 \ | \ \exists g\in G_1,  \mathrm{s.t.} \  s(g)=x, t(g)=y \}
\end{gather*}
determines an equivalence relation. The set of equivalence classes

\begin{gather*}
    |\mathcal{G}|=G_0/\sim
\end{gather*}
is called the \emph{coarse moduli space} of $\mathcal{G}$. If $\mathcal{G}$ is a topological groupoid, ${|\mathcal{G}|}$ is a topological space.

\begin{example}
    Let $X$ be a topological space. Consider a category $\mathcal{G}$ defined by ${G_0=G_1=X}$, meaning that the only morphisms in $\mathcal{G}$ are the identity morphisms. Under this setting, $\mathcal{G}$ is a topological groupoid.
\end{example}

\begin{example}
    Let $X$ be a topological space, and let $H$ be a topological group acting on $X$ from the left. Then, we can define a topological groupoid ${H\ltimes X}$, called the \emph{action groupoid} associated with the group action, by setting its space of objects and morphisms as ${(H\ltimes X)_0=X}$ and ${(H\ltimes X)_1=H\times X}$, respectively. The structural maps are explicitly defined as follows:
\begin{gather*}
    s(h,x)=x, \ t(h,x)=hx, \ m\big((k,hx),(h,x)\big)=(kh,x),  \\
    u(x)=(e,x), \ i(h,x)=(h^{-1},hx).
\end{gather*}
\end{example}

\begin{example}
    Let $X$ be a topological space, and let ${\{U_{\alpha}\}_{\alpha\in \Lambda}}$ be an open covering of $X$. We can construct a topological groupoid $\mathcal{G}$ from this covering by defining the space of objects $G_0$ as the disjoint union of the open sets:
    \begin{gather*}
        G_0=\coprod_{\alpha\in \Lambda}U_{\alpha}.
    \end{gather*}
    The space of morphisms $G_1$ is defined as the disjoint union of fiber products (intersections) over $X$:
    \begin{gather*}
        G_1=\coprod_{\alpha,\beta\in\Lambda}(U_{\alpha}\times_XU_{\beta}),
    \end{gather*}
    where an element ${(x,x)\in U_{\alpha}\times_X U_{\beta}}$ (with ${x\in U_{\alpha}\cap U_{\beta}}$) is regarded as a morphism from ${x\in U_{\alpha}}$ to ${x\in U_{\beta}}$. Under this construction, $\mathcal{G}$ becomes a topological groupoid.
\end{example}

\subsection{Functors, Morita equivalence, and orbifolds}
Since a Lie groupoid can be viewed as a category enriched with a differential structure, a natural next step is to study functors and natural transformations that preserve this geometric nature. In this subsection, we introduce smooth functors and smooth natural transformations, which will lead us to the geometric notion of Morita equivalence.

\begin{definition}
    Let $\mathcal{G}$ and $\mathcal{H}$ be Lie groupoids. A \emph{smooth functor} ${\phi\colon \mathcal{G}\rightarrow \mathcal{H}}$ consists of a pair of smooth maps ${\phi_0\colon G_0\rightarrow H_0}$ and ${\phi_1\colon G_1\rightarrow H_1}$ that are compatible with all structural maps. Explicitly, $\phi$ must satisfy the following conditions for all ${x\in G_0}$ and ${(g,f)\in G_1\times_{s,t}G_1}$:
    \begin{gather*}
        s_{\mathcal{H}}\big(\phi_1(g)\big)=\phi_0\big(s_{\mathcal{G}}(g)\big), \ \ t_{\mathcal{H}}\big(\phi_1(g)\big)=\phi_0\big(t_{\mathcal{G}}(g)\big), \\
        \phi_1\big(u_{\mathcal{G}}(x)\big)=u_{\mathcal{H}}\big(\phi_0(x)\big), \\
        \phi_1\big(m_{\mathcal{G}}(g,f)\big)=m_{\mathcal{H}}\big(\phi_1(g),\phi_1(f)\big),
    \end{gather*}
    where ${s_{\mathcal{G}}, t_{\mathcal{G}},u_{\mathcal{G}},m_{\mathcal{G}}}$ and ${s_{\mathcal{H}}, t_{\mathcal{H}},u_{\mathcal{H}},m_{\mathcal{H}}}$ denote the structural maps of $\mathcal{G}$ and $\mathcal{H}$, respectively.
\end{definition}

\begin{definition}
    Let $\mathcal{G}$ and $\mathcal{H}$ be Lie groupoids, and let ${\phi,\phi'\colon \mathcal{G}\rightarrow \mathcal{H}}$ be two smooth functors. A \emph{smooth natural transformation} ${\alpha\colon \phi\Rightarrow \phi'}$ is a smooth map ${\alpha\colon G_0\rightarrow H_1}$ that assigns to each object ${x\in G_0}$ a morphism ${\alpha(x)\colon \phi(x)\rightarrow \phi'(x)}$ in ${\mathcal{H}}$  such that for any morphism ${g\in G_1}$ from $x$ to $y$ (i.e. ${s_{\mathcal{G}}(g)=x}$ and ${t_{\mathcal{G}}(g)=y}$), the following diagram commutes:
    \[
    \begin{tikzcd}
\phi_0(x) \arrow[rrr, "\alpha(x)"] \arrow[dd, "\phi_1(g)"] &  &  & \phi'_0(x) \arrow[dd, "\phi'_1(g)"] \\
                                                           &  &  &                                   \\
\phi_0(y) \arrow[rrr, "\alpha(y)"]                         &  &  & \phi_0'(y).                        
\end{tikzcd}
    \]
    Explicitly, this commutativity means that $\alpha$ must satisfy the following algebraic relations:
    \begin{enumerate}
        \item ${s_{\mathcal{H}}\big(\alpha(x)\big)=\phi_0(x)}$ and  ${t_{\mathcal{H}}\big(\alpha(x)\big)=\phi'_0(x)}$ for all ${x\in G_0}$,
        \item ${m_{\mathcal{H}}\big(\phi'_1(g),\alpha(x)\big)=m_{\mathcal{H}}\big(\alpha(y),\phi_1(g)\big)}$ for all ${g\in G_1}$ from $x$ to $y$.
    \end{enumerate}
\end{definition}

While smooth functors provide a natural way to relate Lie groupoids, they turn out to be too rigid to capture the correct geometric maps between the underlying spaces. In geometry, two different Lie groupoids can present the same underlying geometric object. To resolve this issue, we need a flexible notion of ``maps" between Lie groupoids, which is achieved by inverting a special class of functors called equivalences. This process corresponds to the localization of the category of Lie groupoids. With this motivation in mind, let us first define what it means for a smooth functor to be an equivalence.

\begin{definition}
    A smooth functor ${\phi\colon \mathcal{G}\rightarrow\mathcal{H}}$ between two Lie groupoids is called an \emph{equivalence} if it satisfies the following conditions:
\begin{enumerate}
    \item The map 
    \begin{gather*}
     G_0\times_{\phi_0,H_0,t_{\mathcal{H}}}H_1\longrightarrow H_0 \\
            (x,h)\longmapsto s_{\mathcal{H}}(h)
    \end{gather*}
     is a surjective submersion.
    \item The diagram
    \[
    \begin{tikzcd}
     G_1 \arrow[rrr, "\phi_1"] \arrow[dd, "s_{\mathcal{G}}\times t_{\mathcal{G}}"] &  &  & H_1 \arrow[dd, "s_{\mathcal{H}}\times t_{\mathcal{H}}"] \\ &  &  &    \\
      G_0\times G_0 \arrow[rrr, "\phi_0\times \phi_0"]                              &  &  & H_0\times H_0                                          
    \end{tikzcd}
     \]
        is a fiber product. Explicitly, the map
        \begin{gather*}
            G_1\longrightarrow (G_0\times G_0)\times_{H_0\times H_0}H_1 \\
            g\longmapsto \big(s_{\mathcal{G}}(g),t_{\mathcal{G}}(g),\phi_1(g)\big)
        \end{gather*}
    \end{enumerate}
    is a diffeomorphism.
\end{definition}
The first condition implies that $\phi$ is essentially surjective, and the second condition implies that $\phi$ is fully faithful. The notion of equivalence defined above plays a fundamental role in constructing the correct category of Lie groupoids. A precise construction of the localization of the category of Lie groupoids with respect to these equivalences, along with the detailed properties of the resulting generalized morphisms, will be explained in Section 3.

Using the notion of equivalence, we can now define a natural equivalence relation on Lie groupoids, known as \emph{Morita equivalence}. Geometrically, Morita equivalent Lie groupoids can be thought of as different groupoid representations of the same underlying geometric object (such as a stack or an orbifold). 

\begin{definition}
    Two Lie groupoids $\mathcal{G}$ and $\mathcal{H}$ are said to be \emph{Morita equivalent} if there exists a third Lie groupoid $\mathcal{K}$ and a pair of equivalences,
    \begin{gather*}
        \mathcal{G}\stackrel{\phi}{\longleftarrow}\mathcal{K}\stackrel{\psi}{\longrightarrow}\mathcal{H}
    \end{gather*}
    where $\phi$ and $\psi$ are smooth functors that are equivalences in the sense of Definition 5.
\end{definition}
It is a standard fact that Morita equivalence defines a true equivalence relation on Lie groupoids. The transitivity of this relation is typically proven by constructing the fiber product of the diagram.

To conclude this subsection, we introduce two important classes of Lie groupoids, which allow us to provide a modern, groupoid-theoretic definition of orbifolds.

\begin{definition}
    A Lie groupoid $\mathcal{G}$ is said to be:
    \begin{enumerate}
        \item \emph{\'{e}tale} if the source map ${s\colon G_1\rightarrow G_0}$ (and hence the target map $t$) is a local diffeomorphism,
        \item \emph{proper} if the map ${(s,t)\colon G_1\rightarrow G_0\times G_0}$, given by ${g\mapsto \big(s(g),t(g)\big)}$, is a proper map in the topological sense (i.e., the preimage of any compact subset is compact).
    \end{enumerate}
\end{definition}
A Lie groupoid that is both proper and \'{e}tale is called a \emph{proper \'{e}tale Lie groupoid}, or simply an \emph{orbifold groupoid}. Historically, orbifolds were defined using local charts and finite group actions. In the modern language of groupoids, they can be defined as the geometric objects represented by orbifold groupoids.

\begin{definition}
    An \emph{orbifold} is a Morita equivalence class of orbifold groupoids.
\end{definition}
Alternatively, one can view an orbifold as a topological space equipped with such an orbifold structure. An \emph{orbifold structure} on a topological space $X$ is given by an orbifold groupoid $\mathcal{G}$ together with a homeomorphism
\begin{gather*}
    \phi\colon|\mathcal{G}|\stackrel{\cong}{\longrightarrow} X
\end{gather*}
where ${|\mathcal{G}|=G_0/\sim}$ denotes the coarse moduli space of $\mathcal{G}$. In this context, we say that the topological space $X$ is endowed with the orbifold structure $\mathcal{G}$.

\section{Localization of Lie groupoids and bibundles}

In Section 2.2, we motivated the need to invert equivalences in order to obtain the correct geometric category of Lie groupoids, where Morita equivalent groupoids become genuinely isomorphic. The algebraic machinery for such a process is provided by the general theory of the localization of a category. 

A concrete and geometric description of these localized morphisms can be realized using bibundles (or Hilsum--Skandalis morphisms), which were originally introduced by Hilsum and Skandalis \cite{HS}.

Following the modern and self-contained approach of Lerman \cite{L}, in this section we fix the notation and review the explicit equivalence between the localized $1$-category of Lie groupoids and the category of bibundles. This geometric framework will serve as a fundamental tool for our main results in the subsequent sections.

\subsection{Abstract localization of categories}

In this subsection, we briefly recall the general algebraic construction of the localization of a category. For a detailed treatment and a proof of the existence of such localizations, we refer the reader to Gelfand and Manin \cite{GM}. Let $\mathcal{C}$ be a category, and let ${W\subset \mathrm{Mor}(\mathcal{C})}$ be a collection of morphisms in $\mathcal{C}$. A localization of $\mathcal{C}$ with respect to $W$ is a category ${\mathcal{C}[W^{-1}]}$ together with a functor 
\begin{gather*}
    Q\colon \mathcal{C}\longrightarrow \mathcal{C}[W^{-1}]
\end{gather*}
that satisfies the following properties:

\begin{enumerate}
    \item For any ${w\in W}$, ${Q(w)}$ is an isomorphism in ${\mathcal{C}[W^{-1}]}$,
    \item\textbf{(Universal property)} For any category $\mathcal{D}$ and any functor ${F\colon \mathcal{C}\rightarrow \mathcal{D}}$ such that ${F(w)}$ is an isomorphism for any ${w\in W}$, there exists a unique functor ${\widetilde{F}\colon \mathcal{C}[W^{-1}]\rightarrow \mathcal{D}}$ making the following diagram commute:
\end{enumerate}

\[
\begin{tikzcd}
\mathcal{C} \arrow[rrr, "F"] \arrow[dd, "Q"']         &  &  & \mathcal{D} \\
                                                      &  &  &             \\
{\mathcal{C}[W^{-1}].} \arrow[rrruu, "\widetilde{F}"'] &  &  &            
\end{tikzcd}
\]

Explicitly, the existence of the localized category ${\mathcal{C}[W^{-1}]}$ is established by introducing a formal inverse for each ${w\in W}$. One first considers a category generated by the original morphisms of $\mathcal{C}$ together with a new morphism (formal inverse) ${x_{w}\colon Y\rightarrow X}$ for each ${w\colon X\rightarrow Y}$ in $W$. A morphism from $X$ to $Y$ in this localized category is represented by an equivalence class of a finite path of the form:

\begin{gather*}
    X=X_0\stackrel{\alpha_1}{\longrightarrow}X_1\stackrel{\alpha_2}{\longrightarrow}X_2\stackrel{\alpha_3}{\longrightarrow}\cdots \stackrel{\alpha_n}{\longrightarrow}X_n=Y
\end{gather*}
where each arrow ${\alpha_i\colon X_{i-1}\rightarrow X_i}$ is either an original morphism in $\mathcal{C}$ or a formal inverse ${x_w}$ of ${w\in W}$. The set ${\mathrm{Hom}_{\mathcal{C}[W^{-1}]}(X,Y)}$ is defined as the set of all such paths modulo the equivalence relation generated by the following relations:

\begin{enumerate}
    \item For each ${w\colon X\rightarrow Y}$ in $W$, the paths
    \begin{gather*}
        X\stackrel{w}{\longrightarrow}Y\stackrel{x_w}{\longrightarrow}X \quad \mathrm{and} \quad Y\stackrel{x_w}{\longrightarrow}X\stackrel{w}{\longrightarrow}Y
    \end{gather*}
are equivalent to
\begin{gather*}
    X\stackrel{\mathrm{Id}_X}{\longrightarrow} X \quad \mathrm{and} \quad Y\stackrel{\mathrm{Id}_Y}{\longrightarrow} Y
\end{gather*}
respectively.
\item If $f,g\in \mathrm{Mor}(\mathcal{C})$ are composable, then
\begin{gather*}
    X\stackrel{f}{\longrightarrow}Y\stackrel{g}{\longrightarrow}Z
\end{gather*}
is equivalent to 
\begin{gather*}
    X\stackrel{g\circ f}{\longrightarrow}Z.
\end{gather*}
\end{enumerate}
Although this algebraic construction using formal inverses and equivalence relations always works, computing ${\mathrm{Hom}_{\mathcal{C}[W^{-1}]}(X,Y)}$ in practice is difficult. 

We now apply this abstract algebraic framework to the geometric setting of Lie groupoids. Let ${\mathbf{Gpoid}}$ be the category whose objects are Lie groupoids and whose morphisms are smooth functors. From a geometric point of view, isomorphic smooth functors should be identified.

\begin{definition}
    We define ${\mathbf{Gp}}$ as the category whose objects are Lie groupoids, and whose morphisms are isomorphism classes of smooth functors. Here, two smooth functors ${F,G\colon \mathcal{G}\rightarrow\mathcal{H}}$ are defined to be isomorphic if there exists a smooth natural transformation between them. Note that since $\mathcal{H}$ is a Lie groupoid, any such natural transformation is automatically an isomorphism, since every morphism in ${H_1}$ has a smooth inverse.
\end{definition}

The category ${\mathbf{Gp}}$ still lacks another fundamental geometric refinement. Equivalences are not yet genuinely invertible within ${\mathbf{Gp}}$. If two Lie groupoids ${\mathcal{G}}$ and ${\mathcal{H}}$ are Morita equivalent, they describe exactly the same geometric object, much like choosing different atlases on a single manifold. Therefore, from a purely geometric perspective, two Morita equivalent Lie groupoids must be isomorphic in the category of Lie groupoids. In order to force equivalences to be invertible, we perform the localization process. We define ${W\subset \mathrm{Mor}(\mathbf{Gp})}$ as the set of all isomorphism classes of equivalences between Lie groupoids. By formally inverting all elements in $W$, we obtain a localized category of Lie groupoids ${\mathbf{Gp}[W^{-1}]}$. 

Consequently, the localized category ${\mathbf{Gp}[W^{-1}]}$ provides a geometrically correct framework for Lie groupoids within the context of $1$-categories. To describe this localized category in a more concrete and geometric manner, we are naturally led to the introduction of bibundles in the next subsection.

\subsection{A geometric model for the localized category}

In this subsection, we introduce the notion of bibundles to provide a differential geometric model for the morphisms in the localized category ${\mathbf{Gp}[W^{-1}]}$. First, we begin by recalling the definition of an action of a Lie groupoid on a manifold.

\begin{definition}
    A right action of a Lie groupoid ${\mathcal{H}}$ on a manifold $P$ consists of the following data:
    \begin{enumerate}
        \item (anchor map) \ ${a\colon P\rightarrow H_0}$.
        \item (right action) \ ${P\times_{a,H_0,t}H_1\rightarrow P}$, \ ${(p,h)\mapsto p\cdot h}$.
    \end{enumerate}
    Moreover, they satisfy the following conditions:
    \begin{itemize}
        \item ${a(p\cdot h)=s(h)}$ holds for all ${(p,h)\in P\times_{a,H_0,t}H_1}$.
        \item ${(p\cdot h_1)\cdot h_2=p\cdot (h_1h_2)}$ holds for all ${p\in P}$, ${h_1,h_2\in H_1}$ with ${a(p)=t(h_1)}$ and ${s(h_1)=t(h_2)}$.
        \item ${p\cdot 1_{a(p)}=p}$ holds for all ${p\in P}$.
    \end{itemize}
\end{definition}

We use the right action to formulate the notion of a principal right ${\mathcal{H}}$-bundle.

\begin{definition}
    Let ${\mathcal{H}}$ be a Lie groupoid, and let $P$ and $B$ be smooth manifolds. A smooth manifold $P$ equipped with a right $\mathcal{H}$-action and a surjective submersion ${\pi\colon P\rightarrow B}$ is a \emph{principal (right)} ${\mathcal{H}}$\emph{-bundle} if it satisfies the following conditions:
    \begin{enumerate}
        \item ${\pi(p\cdot h)=\pi(p)}$ holds for all ${(p,h)\in P\times_{a,H_0,t}H_1}$.
        \item ${P\times_{a,H_0,t}H_1\rightarrow P\times_BP}$, ${(p,h)\mapsto (p,p\cdot h)}$ is a diffeomorphism. So, $\mathcal{H}$ acts freely and transitively on the fiber of ${\pi\colon P\rightarrow B}$.
    \end{enumerate}
\end{definition}
Let $\mathcal{H}$ be a Lie groupoid. Then, $H_1$ naturally becomes a principal right $\mathcal{H}$-bundle over $H_0$. The anchor map ${a\colon H_1\rightarrow H_0}$ is given by the source map $s$, and the projection ${\pi\colon H_1\rightarrow H_0}$ is given by the target map $t$. This example is called the \emph{unit principal} $\mathcal{H}$\emph{-bundle}.

As in the case of ordinary principal bundles, principal right $\mathcal{H}$-bundles admit pull-back constructions along smooth maps between base manifolds. Let ${\pi\colon P\rightarrow B}$ be a principal right $\mathcal{H}$-bundle, and let ${f\colon M\rightarrow B}$ be a smooth map. We can define the pull-back of $P$, denoted by ${f^*P}$, as the fiber product of $\pi$ and $f$:
\[
    f^*P=M\times_BP.
\]
This ${f^*P}$ naturally inherits a canonical right $\mathcal{H}$-action and becomes a principal right $\mathcal{H}$-bundle.

Similarly to the right action, we can define the left action of a Lie groupoid.

\begin{definition}
    A left action of a Lie groupoid $\mathcal{G}$ on a manifold $P$ consists of the following data:
    \begin{enumerate}
        \item (left anchor map) \ ${a\colon P\rightarrow G_0}$.
        \item (left action) \ ${G_1\times_{s,G_0,a}P\rightarrow P}$, ${(g,p)\mapsto g\cdot p}$.
    \end{enumerate}
    We assume that they satisfy the following conditions:
    \begin{itemize}
        \item ${a(g\cdot p)=t(g)}$ holds for all ${(g,p)\in G_1\times_{s,G_0,a}P}$.
        \item ${g_2\cdot (g_1\cdot p)=(g_2g_1)\cdot p}$ holds for all ${p\in P}$, ${g_1,g_2\in G_1}$ with ${a(p)=s(g_1)}$ and ${t(g_1)=s(g_2)}$.
        \item ${1_{a(p)}\cdot p=p}$ holds for all ${p\in P}$.
    \end{itemize}
\end{definition}

A principal left $\mathcal{G}$-bundle is defined in a completely analogous way by replacing a right $\mathcal{H}$-action with a left $\mathcal{G}$-action. A smooth surjective submersion ${\pi\colon P\rightarrow B}$ equipped with a left $\mathcal{G}$-action with respect to a left anchor map ${a\colon P\rightarrow G_0}$ is called a \emph{principal left} $\mathcal{G}$\emph{-bundle} over $B$ if $\pi$ is ${\mathcal{G}}$-invariant and the map
\begin{gather*}
    G_1\times_{s,G_0,a}P\rightarrow P\times_BP, \ (g,p)\mapsto (g\cdot p,p)
\end{gather*}
is a diffeomorphism.

Before proceeding to the definition of bibundles, we provide a natural geometric motivation arising from smooth functors between Lie groupoids. Let ${f\colon \mathcal{G}\rightarrow \mathcal{H}}$ be a smooth functor between Lie groupoids. Recall that $H_1$ naturally forms a principal right $\mathcal{H}$-bundle over $H_0$ (unit principal right $\mathcal{H}$-bundle). We can pull back this principal right $\mathcal{H}$-bundle by the map ${f_0\colon G_0\rightarrow H_0}$:

\[
    f_0^*H_1=G_0\times_{f_0,H_0,t}H_1.
\]
The smooth functor $f$ allows us to equip ${f_0^*H_1}$ with a left $\mathcal{G}$-action that commutes with the right $\mathcal{H}$-action. For all ${\big(g,(x,h)\big)\in G_1\times_{G_0}f_0^*H_1}$ with ${s_{\mathcal{G}}(g)=x}$, we define the left action by
\begin{gather*}
    g\cdot (x,h)=\big(t_{\mathcal{G}}(g),f_1(g)\cdot h\big).
\end{gather*}
Consequently, we see that any smooth functor ${f\colon \mathcal{G}\rightarrow \mathcal{H}}$ yields a manifold equipped simultaneously with a principal right $\mathcal{H}$-action and a left $\mathcal{G}$-action. This structure serves as a prototypical example of a bibundle.

\begin{definition}
    Let $\mathcal{G}$ and ${\mathcal{H}}$ be Lie groupoids. A \emph{bibundle} from $\mathcal{G}$ to $\mathcal{H}$ is a manifold $P$ with two smooth anchor maps ${a_L\colon P\rightarrow G_0}$, ${a_R\colon P\rightarrow H_0}$ satisfying the following conditions: 
\begin{enumerate}
    \item There is a left action of $\mathcal{G}$ on $P$ with respect to the left anchor map ${a_L}$ and a right action of $\mathcal{H}$ on $P$ with respect to the right anchor map $a_R$.
    \item ${a_L\colon P\rightarrow G_0}$ is a principal right $\mathcal{H}$-bundle.
    \item ${a_R(g\cdot p)=a_R(p)}$ holds for all ${(g,p)\in G_1\times_{G_0}P}$.
    \item The left $\mathcal{G}$-action and the right ${\mathcal{H}}$-action are commutative:
    \begin{gather*}
        g\cdot(p\cdot h)=(g\cdot p)\cdot h.
    \end{gather*}
\end{enumerate}
\end{definition}
As we shall see later, such a bibundle $P$ plays the role of a morphism from $\mathcal{G}$ to $\mathcal{H}$. Thus, we write ${P\colon \mathcal{G}\rightarrow \mathcal{H}}$ to treat a bibundle as a morphism. Two bibundles ${P,Q\colon \mathcal{G}\rightarrow \mathcal{H}}$ are called \emph{isomorphic} if there exists a ${\mathcal{G}}$-$\mathcal{H}$ equivariant diffeomorphism between them.

As we have seen, a smooth functor ${f\colon \mathcal{G}\rightarrow \mathcal{H}}$ determines a bibundle ${f_0^*H_1}$. For simplicity, we denote this bibundle by ${\langle f\rangle}$. Suppose that we are given two smooth functors ${f,g\colon \mathcal{G}\rightarrow \mathcal{H}}$. If there exists a smooth natural transformation ${\alpha\colon f\Rightarrow g}$ between them, the bibundles ${\langle f\rangle }$ and ${\langle g\rangle }$ are canonically isomorphic. Indeed, we define a map ${\Phi\colon \langle f\rangle \rightarrow \langle g\rangle}$ by ${\Phi(x,h)=(x,\alpha(x)h)}$. It is straightforward to verify that $\Phi$ is a $\mathcal{G}$-$\mathcal{H}$ equivariant diffeomorphism. Next, we consider the composition of two bibundles, which corresponds to the composition of two maps.

\begin{definition}
    Let $\mathcal{G}$, $\mathcal{H}$ and $\mathcal{K}$ be Lie groupoids, and let ${P\colon \mathcal{G}\rightarrow \mathcal{H}}$ and ${Q\colon \mathcal{H}\rightarrow \mathcal{K}}$ be bibundles. We consider a right $\mathcal{H}$-action on the fiber product ${P\times_{H_0}Q}$ defined by ${(p,q)\cdot h=(p\cdot h,h^{-1}\cdot q)}$. Then, the quotient
    \begin{gather*}
        Q\circ P=(P\times_{H_0}Q)/ \mathcal{H}
    \end{gather*}
    is a smooth manifold and naturally inherits the structure of a bibundle from $\mathcal{G}$ to ${\mathcal{K}}$. We call this bibundle ${Q\circ P\colon \mathcal{G}\rightarrow \mathcal{K}}$ the composition of $P$ and $Q$.
\end{definition}
When considering the composition of three bibundles $P_1$, $P_2$ and $P_3$, this composition is not strictly associative. Indeed, although ${(P_3\circ P_2)\circ P_1}$ and ${P_3\circ (P_2\circ P_1)}$ are canonically isomorphic, they are not strictly equal. Consequently, when considering a category with bibundles as morphisms, we obtain not a strict $2$-category, but a bicategory (or weak $2$-category). However, since the higher-categorical machinery is not needed for the main purpose of this paper, we shall not pursue this direction further. 

\begin{remark}
Suppose that a bibundle ${P\colon \mathcal{G}\rightarrow \mathcal{H}}$ is also a principal left $\mathcal{G}$-bundle. Then we can define a new bibundle ${P^{-1}\colon \mathcal{H}\rightarrow \mathcal{G}}$. Here, $P^{-1}$ coincides with $P$ as a smooth manifold, but the roles of the anchor maps are swapped. The left $\mathcal{G}$-action and the right ${\mathcal{H}}$-action on $P$ are replaced by a right ${\mathcal{G}}$-action and a left $\mathcal{H}$-action, respectively, as follows:

\begin{gather*}
    p\cdot_{P^{-1}}g=g^{-1}\cdot_Pp, \ \ \ h\cdot_{P^{-1}}p=p\cdot_Ph^{-1}.
\end{gather*}
Under this construction, we have canonical isomorphisms
\begin{gather*}
    P^{-1}\circ P\cong \langle \mathrm{Id}_{\mathcal{G}}\rangle, \ \ \ P\circ P^{-1}\cong \langle \mathrm{Id}_{\mathcal{H}}\rangle.
\end{gather*}
\end{remark}
Bibundles that are principal left $\mathcal{G}$-bundles are related to equivalences between Lie groupoids. Indeed, a smooth functor ${f\colon \mathcal{G}\rightarrow \mathcal{H}}$ is an equivalence (i.e., essentially surjective and fully faithful) if and only if the induced bibundle ${\langle f\rangle \colon \mathcal{G}\rightarrow \mathcal{H}}$ is a principal left $\mathcal{G}$-bundle (Lemma 3.34 in \cite{L}). Let ${\mathbf{HS}}$ be the Hilsum--Skandalis category, which is a $1$-category whose objects are Lie groupoids and whose morphisms are isomorphism classes of bibundles. The assignment ${f\mapsto \langle f\rangle}$ defines a functor
\begin{gather*}
    \widetilde{z}\colon \mathbf{Gpoid}\longrightarrow \mathbf{HS}
\end{gather*}
from the category ${\mathbf{Gpoid}}$ of Lie groupoids and smooth functors to ${\mathbf{HS}}$. Since smooth natural transformations induce isomorphisms of bibundles, $\widetilde{z}$ naturally induces a functor
\begin{gather*}
    z\colon \mathbf{Gp}\longrightarrow \mathbf{HS}
\end{gather*}
where morphisms in ${\mathbf{Gp}}$ are isomorphism classes of smooth functors. As before, let $W$ be the set of morphisms in ${\mathbf{Gp}}$ represented by equivalences between Lie groupoids.

\begin{gather*}
    W=\{ \ [w]  \mid  w\in \mathbf{Gpoid}_1  \text{ is an equivalence} \ \}
\end{gather*}
Since the bibundle ${\langle f\rangle \colon \mathcal{G}\rightarrow \mathcal{H}}$ induced by any equivalence ${f\colon \mathcal{G}\rightarrow \mathcal{H}}$ is invertible in ${\mathbf{HS}}$, the universal property of localization induces a functor 
\begin{gather*}
    \widehat{z}\colon \mathbf{Gp}[W^{-1}]\longrightarrow \mathbf{HS}.
\end{gather*}
In this paper, we use bibundles to describe mapping spaces between orbifolds, and the following theorem justifies our choice.

\begin{theorem}[Proposition 3.39 in \cite{L}]
    The functor
    \begin{gather*}
       \widehat{z}\colon \mathbf{Gp}[W^{-1}]\longrightarrow \mathbf{HS} 
    \end{gather*}
    is an equivalence of categories.
\end{theorem}
The localized category ${\mathbf{Gp}[W^{-1}]}$ serves as the natural $1$-category of Lie groupoids. By the above theorem, it can be identified with the Hilsum--Skandalis category ${\mathbf{HS}}$. With this identification in hand, our main goal is to endow the ``mapping space" from an orbifold $\mathcal{G}$ to another orbifold ${\mathcal{H}}$ with a natural orbifold structure, in such a way that its coarse moduli space naturally coincides with the set of isomorphism classes of bibundles from $\mathcal{G}$ to $\mathcal{H}$.

\section{Orbifold mapping spaces}
In this section, we construct infinite-dimensional orbifold structures on mapping spaces between orbifolds. First, we introduce the notion of admissible orbifold groupoids. Throughout this section, we assume that the coarse moduli space of an orbifold groupoid is compact unless otherwise stated.

\begin{definition}
A pair ${(\mathcal{G},\widehat{\mathcal{G}})}$ of orbifold groupoids is called an \emph{admissible pair} if there exists a smooth equivalence ${f\colon\mathcal{G}\rightarrow \widehat{\mathcal{G}}}$ satisfying the following properties:    
\begin{enumerate}
    \item The maps ${f_0\colon G_0\rightarrow \widehat{G}_0}$ and ${f_1\colon G_1\rightarrow \widehat{G}_1}$ are open embeddings.
    \item ${f_0(G_0)}$ is a relatively compact subset of ${\widehat{G}_0}$.
     \item ${f_1(G_1)}$ is a relatively compact subset of ${\widehat{G}_1}$.
\end{enumerate}
In this case, we call $\mathcal{G}$ itself an \emph{admissible orbifold groupoid}.
\end{definition}
For an admissible pair ${(\mathcal{G},\widehat{\mathcal{G}})}$, we identify $\mathcal{G}$ with its image ${f(\mathcal{G})=(f_0(G_0),f_1(G_1))}$ without explicit mention.

\begin{proposition}
    For any orbifold groupoid $\mathcal{G}$, there exists an admissible orbifold groupoid $\mathcal{H}$ that is Morita equivalent to $\mathcal{G}$.
\end{proposition}
For an orbifold groupoid $\mathcal{G}$, we fix a ${G_1}$-invariant Riemannian metric on $G_0$. The existence of such a metric follows from a gluing of locally defined invariant Riemannian metrics via a partition of unity. Moreover, the choice of such a metric is non-essential in this section because the coarse moduli space ${|\mathcal{G}|}$ is assumed to be compact.

\begin{proof}
 For any point ${p\in G_0}$ and any ${\delta>0}$, we define ${T_pG_0^{<\delta}}$ by
 \begin{gather*}
     T_pG_0^{<\delta}=\{v\in T_pG_0 \ | \ |v|<\delta\}.
 \end{gather*}
 We then choose a sufficiently small ${\epsilon_p>0}$ so that the exponential map
 \begin{gather*}
     \exp\colon T_pG_0^{<3\epsilon_p}\longrightarrow G_0
 \end{gather*}
 is a diffeomorphism onto its image. For ${\delta \le 3\epsilon_p}$, we denote ${\exp\big(T_pG_0^{<\delta}\big)}$ by ${U_p^{\delta}}$. By replacing $\epsilon_p$ with a sufficiently small value, we may assume that the restriction of $\mathcal{G}$ to ${U_p^{3\epsilon_p}}$ is isomorphic to the action groupoid ${K_p\ltimes U_p^{3\epsilon_p}}$ where $K_p$ is the isotropy group of $p$. Note that the restriction of $\mathcal{G}$ to ${U_p^{\delta}}$ is isomorphic to the action groupoid ${K_p\ltimes U_p^{\delta}}$ for all ${\delta\le 3\epsilon_p}$. Let ${\pi\colon G_0\rightarrow |\mathcal{G}|}$ be the natural projection. Since the coarse moduli space ${|\mathcal{G}|}$ is compact, we can choose ${p_1,\cdots,p_m\in G_0}$ such that
 \begin{gather*}
     |\mathcal{G}|=\bigcup_{1\le i\le m}\pi(U_{p_i}^{\epsilon_{p_i}}).
 \end{gather*}
We construct a new orbifold groupoid ${\mathcal{H}=(H_0,H_1)}$ as follows. First, we set
\begin{gather*}
    H_0=\bigsqcup_{1\le i\le m}U_{p_i}^{\epsilon_{p_i}}.
\end{gather*}
For any ${x,y\in H_0}$, the set of morphisms from $x$ to $y$ is defined by
\begin{gather*}
    H_1(x,y)=G_1(x,y).
\end{gather*}
Similarly, we construct a new orbifold groupoid ${\widehat{\mathcal{H}}=(\widehat{H}_0,\widehat{H}_1)}$ as follows:
\begin{gather*}
    \widehat{H}_0=\bigsqcup_{1\le i\le m}U_{p_i}^{2\epsilon_{p_i}}, \quad \widehat{H}_1(x,y)=G_1(x,y) \quad \mathrm{for} \ x,y\in \widehat{H}_0.
\end{gather*}
Since the images ${\pi(U_{p_i}^{\epsilon_{p_i}})}$ cover the coarse moduli space ${|\mathcal{G}|}$, the natural inclusions ${\mathcal{H}\rightarrow \mathcal{G}}$ and ${\widehat{\mathcal{H}}\rightarrow \mathcal{G}}$ are equivalences. In particular, the pair of orbifold groupoids ${(\mathcal{H},\widehat{\mathcal{H}})}$ is an admissible pair and ${\mathcal{H}}$ is an admissible orbifold groupoid that is Morita equivalent to $\mathcal{G}$.
 
\end{proof}

Next, we prove the following extension property for bibundles.

\begin{proposition}
    Let ${P\colon\mathcal{G}\rightarrow \mathcal{H}}$ be a bibundle between orbifold groupoids and let ${(\mathcal{H},\widehat{\mathcal{H}})}$ be an admissible pair. Then, there exist a bibundle ${\widehat{P}\colon \mathcal{G}\rightarrow \widehat{\mathcal{H}}}$ and an inclusion ${P\hookrightarrow \widehat{P}}$ such that the anchor maps ${a_L\colon P\rightarrow G_0}$, ${a_R\colon P\rightarrow H_0}$, ${\widehat{a}_L\colon \widehat{P}\rightarrow G_0}$, ${\widehat{a}_R\colon \widehat{P}\rightarrow \widehat{H}_0}$ satisfy ${\widehat{a}_L|_{P}=a_L}$ and ${\widehat{a}_R|_P=a_R}$. Moreover, such a bibundle ${\widehat{P}\colon \mathcal{G}\rightarrow \widehat{\mathcal{H}}}$ is unique up to isomorphism in the sense that if ${\widehat{P}'\colon \mathcal{G}\rightarrow \widehat{\mathcal{H}}}$ is another bibundle that satisfies the above properties, there exists a unique isomorphism ${\Phi\colon \widehat{P}\rightarrow \widehat{P}'}$ that commutes with the inclusion of ${P}$.
\end{proposition}

\begin{proof}
    First, we prove the existence of a right $\widehat{\mathcal{H}}$-bundle near any point ${p\in G_0}$. We choose a sufficiently small open neighborhood $U_p$ of $p$ and a local section ${\rho\colon U_p\rightarrow P}$ of the surjective submersion ${a_L\colon P\rightarrow G_0}$. Then, there exists a natural diffeomorphism
    \begin{gather*}
        \{(q,\gamma)\in U_p\times H_1 \ | \ a_R\circ \rho(q)=t(\gamma)\}\longrightarrow a_L^{-1}(U_p)  \\
        (q,\gamma)\longmapsto \rho(q)\cdot \gamma.
    \end{gather*}
    Moreover, this is an isomorphism as a right $\mathcal{H}$-bundle. We denote the right $\mathcal{H}$-bundle ${a_L^{-1}(U_p)}$ over $U_p$ by ${P|_{U_p}}$. We define a right ${\widehat{\mathcal{H}}}$-bundle ${\widetilde{P}_{U_p}}$ over $U_p$ by
    \begin{gather*}
        \widetilde{P}_{U_p}=\{(q,\gamma)\in U_p\times \widehat{H}_1 \mid a_R\circ \rho(q)=t(\gamma)\}.
    \end{gather*}
    Note that we can regard ${P|_{U_p}}$ as a subset of ${\widetilde{P}_{U_p}}$ via the above isomorphism. In this way, we can extend a right $\mathcal{H}$-bundle to a right $\widehat{\mathcal{H}}$-bundle. Next, for any two points ${p,q\in G_0}$, we show that we can construct a right ${\widehat{\mathcal{H}}}$-bundle over ${U_p\cup U_q}$ by gluing the two extended right ${\widehat{\mathcal{H}}}$-bundles over $U_p$ and $U_q$. For any point ${r\in U_p\cap U_q}$, we fix a point ${r'\in a_L^{-1}(r)\subset P}$. Then, as a pair consisting of a right $\mathcal{H}$-bundle and a right $\widehat{\mathcal{H}}$-bundle over ${\{r\}}$, the inclusion ${P|_{\{r\}}\subset \widetilde{P}_{U_p}|_{\{r\}}}$ (as well as ${P|_{\{r\}}\subset \widetilde{P}_{U_q}|_{\{r\}}}$) is isomorphic to the pair
    \begin{gather*}
        \{\gamma \in H_1 \mid t(\gamma)= a_R(r')\}\subset \{\gamma\in \widehat{H}_1 \mid t(\gamma)= \widehat{a}_R(r')\}.
    \end{gather*}
    Actually, the map
    \begin{gather*}
        \{\gamma\in H_1 \mid t(\gamma)=a_R(r')\}\longrightarrow P|_{\{r\}}  \\
        \gamma \longmapsto r'\cdot \gamma
    \end{gather*}
    is an isomorphism between right $\mathcal{H}$-bundles. For any ${\gamma \in \widehat{H}_1}$ with ${t(\gamma)=a_R(r')}$, there exist a unique ${x\in \widetilde{P}_{U_p}|_{\{r\}}}$ and a unique ${y\in \widetilde{P}_{U_q}|_{\{r\}}}$ corresponding ${r'\cdot \gamma}$:
    \begin{gather*}
        x=r'\cdot \gamma, \quad y=r'\cdot \gamma.
    \end{gather*}
    Thus, there exists a unique isomorphism ${\widetilde{P}_{U_p}|_{\{r\}}\rightarrow \widetilde{P}_{U_q}|_{\{r\}}}$ that induces  the identity on ${P|_{\{r\}}}$. In other words, there is a unique isomorphism between right ${\widehat{\mathcal{H}}}$-bundles over ${r\in G_0}$
    \begin{gather*}
        \phi_{pq}|_{\{r\}}\colon \widetilde{P}_{U_p}|_{\{r\}}\longrightarrow \widetilde{P}_{U_q}|_{\{r\}}
    \end{gather*}
    such that 
    \begin{gather*}
        \big(\phi_{pq}|_{\{r\}}\big)\big|_{P|_{\{r\}}}=\mathrm{Id}_{P|_{\{r\}}}
    \end{gather*}
    holds. This ${\phi_{pq}|_{\{r\}}}$ extends naturally to an isomorphism between right $\widehat{\mathcal{H}}$-bundles over ${U_p\cap U_q}$
    \begin{gather*}
        \phi_{pq}\colon \widetilde{P}_{U_p}|_{U_p\cap U_q}\longrightarrow \widetilde{P}_{U_q}|_{U_p\cap U_q}.
    \end{gather*}
    Therefore, we can construct a right ${\widehat{\mathcal{H}}}$-bundle over ${U_p\cup U_q}$ by gluing ${\widetilde{P}_{U_p}}$ and ${\widetilde{P}_{U_q}}$ via ${\phi_{pq}}$
    \begin{gather*}
        \big(\widetilde{P}_{U_p}\amalg \widetilde{P}_{U_q}\big)\big/ (x\sim \phi_{pq}(x)).
    \end{gather*}
    This yields an extension of $P$ to a right ${\widehat{\mathcal{H}}}$-bundle over ${U_p\cup U_q}$. By applying this gluing construction, we define a right $\widehat{\mathcal{H}}$-bundle $\widehat{P}$ by
    \begin{gather*}
        \widehat{P}=\Big(\coprod_{p\in G_0}\widetilde{P}_{U_p}\Big)\Big/ \sim.
    \end{gather*}
    Here, the equivalence relation $\sim$ is defined by ${x\sim \phi_{pq}(x)}$ for ${x\in \widetilde{P}_{U_p}|_{U_p\cap U_q}}$. We can naturally define a surjective submersion ${\widehat{a}_L\colon \widehat{P}\rightarrow G_0}$ and a right anchor map ${\widehat{a}_R\colon \widehat{P}\rightarrow \widehat{H}_0}$. Furthermore, $\widehat{P}$ contains the original bibundle $P$ as a subset. This $\widehat{P}$ is a desired extension of the right $\mathcal{H}$-bundle $P$ to a right $\widehat{\mathcal{H}}$-bundle. Next, we define a left $\mathcal{G}$-action on $\widehat{P}$. We choose any ${\gamma \in G_1}$ with ${s(\gamma)=p}$, ${t(\gamma)=q}$. Then, the left $\gamma$-action on $P$ induces a bijection that commutes with the right $\mathcal{H}$-action
    \begin{gather*}
        \psi_{\gamma}\colon P|_{\{p\}}\longrightarrow P|_{\{q\}} \\
        x\longmapsto \gamma \cdot x.
    \end{gather*}
    Similarly to the above argument, this bijection uniquely extends to a bijection $\widehat{\psi}_{\gamma}$ that commutes with the right $\widehat{\mathcal{H}}$-action
    \begin{gather*}
        \widehat{\psi}_{\gamma}\colon \widehat{P}|_{\{p\}}\longrightarrow \widehat{P}|_{\{q\}}.
    \end{gather*}
    This determines a left $\mathcal{G}$-action on $\widehat{P}$. Thus, $\widehat{P}$ is a desired bibundle that extends the original bibundle $P$.
    Finally, we prove the uniqueness of the extension of $P$. Let ${\widehat{P}'}$ be another extension of $P$ as a $\mathcal{G}$--${\widehat{\mathcal{H}}}$ bibundle. Similarly to the above argument, for any ${r\in G_0}$, there exists a unique isomorphism ${\Phi_r\colon \widehat{P}|_{\{r\}}\rightarrow \widehat{P}'|_{\{r\}}}$ between right ${\widehat{\mathcal{H}}}$-bundles that commutes with the inclusions ${P|_{\{r\}}\hookrightarrow \widehat{P}|_{\{r\}}}$ and ${P|_{\{r\}}\hookrightarrow \widehat{P}'|_{\{r\}}}$. 
    These local isomorphisms $\Phi_r$ can be glued smoothly over $G_0$ to form a global isomorphism ${\Phi:\widehat{P}\rightarrow \widehat{P}'}$. This completes the proof of uniqueness.
\end{proof}

We introduce the notion of $C^k$ bibundles.

\begin{definition}
    For ${k\in \mathbb{Z}_{\ge0}\cup \{\infty\}}$, a $C^k$ bibundle from a Lie groupoid $\mathcal{G}$ to a Lie groupoid $\mathcal{H}$ is defined in the same way as an ordinary bibundle, with the exception that the right anchor map ${a_R\colon P\rightarrow H_0}$ is a $C^k$ map between smooth manifolds. Note that $C^{\infty}$ bibundles are ordinary bibundles.
\end{definition}

Similarly, we introduce the notion of bibundles of Sobolev class.

\begin{definition}
    Let $\mathcal{G}$, $\mathcal{H}$ be orbifold groupoids and let $k$, $p$ be positive integers such that ${kp>\mathrm{dim}\mathcal{G}}$ holds. A $W^{k,p}$ bibundle from $\mathcal{G}$ to $\mathcal{H}$ is defined in the same way as an ordinary bibundle, with the exception that the right anchor map ${a_R\colon P\rightarrow H_0}$ is a locally $W^{k,p}$ map between smooth manifolds. Note that ${\mathrm{dim}P=\mathrm{dim}\mathcal{G}}$ holds because the fiber of the surjective submersion ${a_L\colon P\rightarrow G_0}$ is discrete. Therefore, a local $W^{k,p}$ map ${a_R}$ is continuous by the Sobolev embedding theorem.
\end{definition}

Next, for any orbifold groupoids $\mathcal{G}$, ${\mathcal{H}}$ and for any ${k\in \mathbb{Z}_{\ge 0}\cup \{\infty\}}$, we prove that the space of $C^k$ bibundles from $\mathcal{G}$ to $\mathcal{H}$ admits an infinite-dimensional orbifold groupoid structure. Without loss of generality, we may assume that $\mathcal{H}$ is an admissible orbifold groupoid by Proposition 18. Let ${(\mathcal{H},\widehat{\mathcal{H}})}$ be an admissible pair, and let ${P}$ be a $C^k$ bibundle from $\mathcal{G}$ to $\mathcal{H}$. To define a topology on this space of bibundles, we first construct a neighborhood of $P$. 

We fix an invariant Riemannian metric on $\widehat{H}_0$. Then, this metric lifts to a Riemannian metric on $\widehat{H}_1$. For any ${p\in H_0}$ and ${r>0}$, let ${T_pH_0^{<r}}$ be the open ball
\begin{gather*}
    T_pH_0^{<r}=\{v\in T_pH_0 \mid |v|<r\},
\end{gather*}
and let ${\delta_0(p)>0}$ be the real number defined by
\begin{gather*}
    \delta_0(p)=\sup \{r>0 \mid \exp_p \colon T_pH_0^{<r}\rightarrow \widehat{H}_0 \ \text{is a diffeomorphism onto its image.}\}.
\end{gather*}
Then, by the relative compactness of ${H_0\subset \widehat{H}_0}$, we have ${\delta_0(\mathcal{H})=\inf_{p\in H_0}\delta_0(p)>0}$. Similarly, for any ${\gamma \in H_1}$, we define ${T_{\gamma}H_1^{<r}}$ and ${\delta_1(\gamma)}$ as follows:

\begin{gather*}
    T_{\gamma}H_1^{<r}=\{v\in T_{\gamma}H_1 \mid |v|<r\}, \\
    \delta_1(\gamma)=\sup \{r>0 \mid \exp_{\gamma}\colon T_{\gamma}H_1^{<r}\rightarrow \widehat{H}_1 \ \text{is a diffeomorphism onto its image.}\}.
\end{gather*}
Then, we also have ${\delta_1(\mathcal{H})=\inf_{\gamma\in H_1}\delta_1(\gamma)>0}$. We fix a positive real number ${0<r<\dfrac{1}{10}\min\{\delta_0(\mathcal{H}),\delta_1(\mathcal{H})\}}$ such that the following conditions hold.

\begin{enumerate}
    \item If there exists ${\gamma\in H_1}$ with ${s(\gamma)=x}$ and ${t(\gamma)=y}$ for distinct points ${x,y\in H_0}$, then
    \begin{gather*}
        \exp_{x}(T_xH_0^{<r})\cap \exp_y(T_yH_0^{<r})=\emptyset.
    \end{gather*}
    \item By the \'{e}taleness of ${\widehat{\mathcal{H}}}$, the source map $s$ and the target map $t$ are local diffeomorphisms. For any ${\gamma\in H^{r}_1\subset \widehat{H}_1}$, the locally defined map ${s\circ t^{-1}}$ induces a diffeomorphism
    \begin{gather*}
        \phi_{\gamma}\colon \exp_{t(\gamma)}(T_{t(\gamma)}\widehat{H}_0^{<r})\longrightarrow \exp_{s(\gamma)}(T_{s(\gamma)}\widehat{H}_0^{<r}).
    \end{gather*}
    Here, ${H_1^{r}}$ is the $r$-neighborhood of $H_1$ in $\widehat{H}_1$.
    \item For any ${x\in H_0}$ and any distinct elements ${\gamma_1,\gamma_2\in \widehat{H}_1}$ with ${t(\gamma_1)=t(\gamma_2)=x}$, the open $r$-balls ${B_r(\gamma_1)}$ and ${B_r(\gamma_2)}$ in $\widehat{H}_1$ are disjoint.
    \end{enumerate}
Let ${P}$ be a ${C^k}$ bibundle from $\mathcal{G}$ to ${\mathcal{H}}$, and let ${\widehat{P}}$ be its extension to $\widehat{\mathcal{H}}$. For ${0<\epsilon\le r}$, we consider a lift ${f\colon P\rightarrow TH_0^{<\epsilon}}$ of the right anchor map ${a_R\colon P\rightarrow H_0}$ such that the following diagram commutes:
\[
\begin{tikzcd}
  &  &  &  & TH_0^{<\epsilon} \arrow[ddd,"\pi"] \\
  &  &  &  &  \\
  &  &  &  &  \\
P \arrow[rrrr, "a_R"] \arrow[rrrruuu, "f"] &  &  &  & H_0                    \end{tikzcd}
\]
where ${\pi\colon TH_0^{<\epsilon}\rightarrow H_0}$ is the canonical projection. We require $f$ to be $\mathcal{G}$-invariant and ${\mathcal{H}}$-equivariant in the following sense. Fix any ${p\in P}$, ${\gamma\in H_1}$ with ${t(\gamma)=a_R(p)}$. The element $\gamma$ induces a diffeomorphism between neighborhoods of ${a_R(p)}$ and ${a_R(p\cdot \gamma)=s(\gamma)}$ in $H_0$. Thus, it yields a linear isomorphism ${(\phi_{\gamma})_*\colon T_{a_R(p)}H_0\rightarrow T_{a_R(p\cdot \gamma)}H_0}$. We say that $f$ is $\mathcal{H}$-equivariant if 
\begin{gather*}
    f(p\cdot\gamma)=(\phi_{\gamma})_*\big(f(p)\big)
\end{gather*}
holds for all such $p$ and $\gamma$. On the other hand, a map $f$ is defined to be $\mathcal{G}$-invariant if
\begin{gather*}
    f(g\cdot p)=f(p)
\end{gather*}
holds for any ${(g,p)\in G_1\times_{G_0}P}$. Such a lift ${f\colon P\rightarrow TH_0^{<\epsilon}}$ extends uniquely to a $\mathcal{G}$-invariant and ${\widehat{\mathcal{H}}}$-equivariant lift ${\widehat{f}\colon \widehat{P}\rightarrow T\widehat{H}_0^{<\epsilon}}$ of the right anchor map ${\widehat{a}_R\colon \widehat{P}\rightarrow \widehat{H}_0}$ such that the following diagram commutes.

\[
\begin{tikzcd}
  &  &  &  & T\widehat{H}_0^{<\epsilon} \arrow[ddd,"\pi"] \\
  &  &  &  &  \\
  &  &  &  &  \\
\widehat{P} \arrow[rrrr, "\widehat{a}_R"] \arrow[rrrruuu, "\widehat{f}"] &  &  &  & \widehat{H}_0                    \end{tikzcd}
\]
Let $P^{\epsilon}$ be an $\epsilon$-neighborhood of $P$ in $\widehat{P}$. We denote the restriction of $\widehat{f}$ to $P^{\epsilon}$ by $f^{\epsilon}$:
\begin{gather*}
    f^{\epsilon}\colon P^{\epsilon}\longrightarrow T\widehat{H}_0^{<\epsilon}.
\end{gather*}
This $f^{\epsilon}$ is $\mathcal{G}$-invariant and ${\mathcal{H}^{\epsilon}}$-equivariant, where ${\mathcal{H}^{\epsilon}=(H_0^{\epsilon},H_1^{\epsilon})}$ is an $\epsilon$-neighborhood of $\mathcal{H}$ in $\widehat{\mathcal{H}}$. We define a map ${\exp_{f^{\epsilon}}}$ by
\begin{gather*}
    \exp_{f^{\epsilon}}\colon P^{\epsilon}\longrightarrow H_0^{2\epsilon} \\
    p\longmapsto \exp_{\widehat{a}
    _R(p)}\big(f^{\epsilon}(p)\big),
\end{gather*}
and we define the set ${P(f)}$ by 
\begin{gather*}
    P(f)=(\exp_{f^{\epsilon}})^{-1}(H_0).
\end{gather*}
We regard ${P(f)}$ as a $C^k$ bibundle that is sufficiently close to $P$. Hereafter, we fix ${0<\epsilon<\dfrac{1}{2}r}$.

\begin{proposition}
    The set ${P(f)}$ admits the structure of a bibundle from $\mathcal{G}$ to $\mathcal{H}$.
\end{proposition}

\begin{proof}
    Note that ${P(f)}$ is an open submanifold of ${\widehat{P}}$. To prove the proposition, it suffices to show the following assertions:
    \begin{itemize}
        \item There exists a left anchor map ${a_{L,P(f)}\colon P(f)\rightarrow G_0}$.
        \item There exists a right anchor map ${a_{R,P(f)}\colon P(f)\rightarrow H_0}$.
        \item There exists a left $\mathcal{G}$-action on ${P(f)}$.
        \item There exists a right $\mathcal{H}$-action $\star$ on ${P(f)}$.
        \item The left anchor map ${a_{L,P(f)}}$ is ${\mathcal{H}}$-invariant.
        \item The right anchor map ${a_{R,P(f)}}$ is $\mathcal{G}$-invariant.
        \item The canonical map ${P(f)\times_{a_{R,P(f)},H_0,t}H_1\rightarrow P(f)\times_{G_0}P(f)}$ given by ${(p,\gamma)\mapsto (p,p\cdot \gamma)}$ is a diffeomorphism.
    \end{itemize}
First, we define a left anchor map ${a_{L,P(f)}\colon P(f)\rightarrow G_0}$ by
\begin{gather*}
    a_{L,P(f)}=\widehat{a}_{L}|_{P(f)}.
\end{gather*}
We also define a left $\mathcal{G}$-action on ${P(f)}$ by restricting the left $\mathcal{G}$-action on $\widehat{P}$ to ${P(f)}$. By the $\mathcal{G}$-invariance of $f$, the left $\mathcal{G}$-action maps ${P(f)}$ into itself. Next, we define a right anchor map ${a_{R,P(f)}}$ by
\begin{gather*}
    a_{R,P(f)}(p)=\exp_{f^{\epsilon}}(p).
\end{gather*}
Fix ${p\in P(f)}$ with ${a_{R,P(f)}(p)=x}$ and ${\gamma\in H_1}$ with ${t(\gamma)=x}$. We want to define a right ${\mathcal{H}}$-action ${p\star \gamma\in P(f)}$. We denote the right $\mathcal{H}$-action on ${P(f)}$ by $\star$ to distinguish it from the right $\mathcal{H}$-action on $P$ and $\widehat{P}$. Recall that there is a diffeomorphism $\phi_{\gamma}$ between $r$-neighborhoods of ${x=t(\gamma)}$ and ${y=s(\gamma)}$. Since ${d(\widehat{a}_{R}(p),a_{R,P(f)}(p))<\epsilon<r}$, Condition ${(2)}$ on $r$ implies that there exists a unique ${\gamma'\in \widehat{H}_1}$ with ${d(\gamma,\gamma')<\epsilon}$ such that ${t(\gamma')=\widehat{a}_R(p)}$. We define 
\begin{gather*}
    p\star \gamma=p\cdot \gamma'.
\end{gather*}
By the ${\mathcal{H}^{\epsilon}}$-equivariance of $f^{\epsilon}$, we have
\begin{gather*}
    a_{R,P(f)}(p\star \gamma)=a_{R,P(f)}(p\cdot \gamma')=\exp_{f^{\epsilon}}(p\cdot \gamma') \\
    =\phi_{\gamma'}\big(\exp_{f^{\epsilon}}(p)\big)=\phi_{\gamma'}(x).
\end{gather*}
Thus, there exists ${\gamma''\in \widehat{H}_1}$ with ${t(\gamma'')=x}$, ${s(\gamma'')=a_{R,P(f)}(p\star \gamma)}$ such that ${d(\gamma,\gamma'')<2\epsilon}$. Since ${2\epsilon<r}$ and ${t(\gamma)=t(\gamma'')=x}$, Condition ${(3)}$ on $r$ guarantees ${\gamma=\gamma''}$. Consequently, we obtain 
\begin{gather*}
    a_{R,P(f)}(p\star \gamma)=s(\gamma).
\end{gather*}
In particular, ${a_{R,P(f)}}$ is $\mathcal{H}$-equivariant. The $\mathcal{H}$-invariance of ${a_{L,P(f)}}$ and the $\mathcal{G}$-invariance of ${a_{R,P(f)}}$ are obvious. Finally, we show that the canonical map
\begin{gather*}
    \psi\colon P(f)\times_{a_{R,P(f)},H_0,t}H_1\longrightarrow P(f)\times_{G_0}P(f) \\
    (p,\gamma)\longmapsto (p,p\star \gamma)
\end{gather*}
is a diffeomorphism. Assume that ${\psi(p_1,\gamma_1)=\psi(p_2,\gamma_2)}$ holds. Then, ${p_1=p_2=p}$ and ${p\star \gamma_1=p\star \gamma_2}$. By the definition of the right ${\mathcal{H}}$-action $\star$, there exist ${\gamma_1',\gamma_2'\in \widehat{H}_1}$ such that
\begin{gather*}
    d(\gamma_1,\gamma_1')<\epsilon, \quad d(\gamma_2,\gamma_2')<\epsilon \\
    p\star \gamma_1=p\cdot \gamma_1', \quad p\star \gamma_2=p\cdot \gamma_2'.
\end{gather*}
Since ${p\cdot \gamma_1'=p\cdot \gamma_2'}$, we get ${\gamma_1'=\gamma_2'}$. Consequently, we obtain
\begin{gather*}
    d(\gamma_1,\gamma_2)\le d(\gamma_1,\gamma_1')+d(\gamma_2,\gamma_2')<2\epsilon<r.
\end{gather*}
Since ${t(\gamma_1)=t(\gamma_2)=a_{R,P(f)}(p)}$, Condition ${(3)}$ on $r$ implies that ${\gamma_1=\gamma_2}$. Therefore, $\psi$ is injective. Let ${(p,q)\in P(f)\times_{G_0}P(f)}$ be any element. By ${(p,q)\in \widehat{P}\times_{G_0}\widehat{P}}$, there exists a unique ${\gamma'\in \widehat{H}_1}$ such that ${q=p\cdot \gamma'}$. Recall ${f^{\epsilon}(q)=(\phi_{\gamma'})\big(f^{\epsilon}(p)\big)}$ holds. Thus, there exists a unique ${\gamma\in \widehat{H}_1}$ in an $\epsilon$-neighborhood of $\gamma'$ such that
\begin{gather*}
    s(\gamma)=a_{R,P(f)}(q),\quad t(\gamma)=a_{R,P(f)}(p).
\end{gather*}
Then ${\gamma\in H_1}$ follows from ${s(\gamma), t(\gamma)\in H_0}$. Therefore, ${q=p\star \gamma}$ holds. This implies that $\psi$ is surjective. Since $\psi$ is a smooth bijection over $G_0$ with discrete fibers and acts as the identity on the base $G_0$, it is a diffeomorphism. This completes the proof.
\end{proof}

For ${k\in \mathbb{Z}_{\ge0}}$ and a $C^k$ bibundle ${P\colon \mathcal{G}\rightarrow \mathcal{H}}$, we define the vector space ${T_P\mathfrak{Map}_0^{C^k}\big(\mathcal{G},(\mathcal{H},\widehat{\mathcal{H}})\big)}$ by
\begin{gather*}
 T_P\mathfrak{Map}_0^{C^k}\big(\mathcal{G},(\mathcal{H},\widehat{\mathcal{H}})\big)=\Bigg\{f\colon P\rightarrow TH_0 \ \Bigg| \ \begin{matrix}
      \pi\circ f=a_R \\ f \ \text{is} \  \mathcal{G}\text{-invariant and} \ \mathcal{H}\text{-equivariant}
   \end{matrix}\Bigg\}.
\end{gather*}
Equipped with the $C^k$-norm, ${T_P\mathfrak{Map}_0^{C^k}\big(\mathcal{G},(\mathcal{H},\widehat{\mathcal{H}})\big)}$ carries a natural structure of a Banach space. For a sufficiently small ${\epsilon>0}$, we define the open $\epsilon$-ball by
\begin{gather*}
T_P\mathfrak{Map}_0^{C^k}\big(\mathcal{G},(\mathcal{H},\widehat{\mathcal{H}})\big)^{<\epsilon}=\Big\{f\in T_P\mathfrak{Map}_0^{C^k}\big(\mathcal{G},(\mathcal{H},\widehat{\mathcal{H}})\big) \ 
\Big| \ |f|_{C^k}<\epsilon  \Big\}.
\end{gather*}
By Proposition 22, for any ${f\in T_P\mathfrak{Map}_0^{C^k}\big(\mathcal{G},(\mathcal{H},\widehat{\mathcal{H}})\big)^{<\epsilon}}$, ${P(f)}$ is a $C^k$ bibundle from $\mathcal{G}$ to $\mathcal{H}$. We define an $\epsilon$-neighborhood of $P$ by
\begin{gather*}
    N_{\epsilon}^{C^k}(P)=\{P(f) \mid f\in T_P\mathfrak{Map}_0^{C^k}\big(\mathcal{G},(\mathcal{H},\widehat{\mathcal{H}})\big)^{<\epsilon}\}.
\end{gather*}
Let ${\operatorname{Iso}^{C^k}(\mathcal{G},\mathcal{H})}$ denote the set of isomorphism classes of $C^k$ bibundles from $\mathcal{G}$ to $\mathcal{H}$. For each class ${[P]\in\operatorname{Iso}^{C^k}(\mathcal{G},\mathcal{H})}$, we fix a representative $C^k$ bibundle $P$. We consider a groupoid ${\mathfrak{Map}^{C^k}\big(\mathcal{G},(\mathcal{H},\widehat{\mathcal{H}})\big)=\big(\mathfrak{Map}_0^{C^k}\big(\mathcal{G},(\mathcal{H},\widehat{\mathcal{H}})\big),\mathfrak{Map}_1^{C^k}\big(\mathcal{G},(\mathcal{H},\widehat{\mathcal{H}})\big)\big)}$ defined by
\begin{gather*}
    \mathfrak{Map}_0^{C^k}\big(\mathcal{G},(\mathcal{H},\widehat{\mathcal{H}})\big)=\coprod_{[P]\in \operatorname{Iso}^{C^k}(\mathcal{G},\mathcal{H})}N_{\epsilon}^{C^k}(P)
    \\
    \mathfrak{Map}_1^{C^k}\big(\mathcal{G},(\mathcal{H},\widehat{\mathcal{H}})\big)=\Bigg\{\Phi\colon P_1\rightarrow P_2 \Bigg| \ \begin{matrix}
        P_1,P_2\in \mathfrak{Map}_0^{C^k}\big(\mathcal{G},(\mathcal{H},\widehat{\mathcal{H}})\big) \\ \Phi \ \text{is a bibundle isomorphism.}
    \end{matrix}\Bigg\}.
\end{gather*}  
Note that ${\widehat{\mathcal{H}}}$ is used in the construction of ${P(f)}$. Later, we show that the equivalence class of this groupoid does not depend on $\widehat{\mathcal{H}}$, thus we can drop ${\widehat{\mathcal{H}}}$ from the notation. For an isomorphism ${\Phi\colon P_1\rightarrow P_2}$ in ${\mathfrak{Map}_1^{C^k}\big(\mathcal{G},(\mathcal{H},\widehat{\mathcal{H}})\big)}$, the source map $s$ and the target map $t$ are given by
\begin{gather*}
    s(\Phi)=P_1, \quad t(\Phi)=P_2.
\end{gather*}

Next, we define the topology ${\mathcal{O}_0^{C^k}\big(\mathcal{G},(\mathcal{H},\widehat{\mathcal{H}})\big)}$ on ${\mathfrak{Map}_0^{C^k}\big(\mathcal{G},(\mathcal{H},\widehat{\mathcal{H}})\big)}$. It is given as the disjoint union topology, so it suffices to specify the topology on each component ${N_{\epsilon}^{C^k}(P)}$. For any ${Q\in N_{\epsilon}^{C^k}(P)}$ and a sufficiently small ${\epsilon_Q>0}$, ${N_{\epsilon_Q}^{C^k}(Q)}$ is canonically identified with a subset of ${N_{\epsilon}^{C^k}(P)}$. We denote this embedding by
\begin{gather*}
    \iota_{Q,P}\colon N_{\epsilon_Q}^{C^k}(Q)\longrightarrow N_{\epsilon}^{C^k}(P).
\end{gather*}
Then, we define the topology ${\mathcal{O}_0^{C^k}\big(\mathcal{G},(\mathcal{H},\widehat{\mathcal{H}})\big)}$ by
\begin{gather*}
    \mathcal{O}_0^{C^k}\big(\mathcal{G},(\mathcal{H},\widehat{\mathcal{H}})\big)=\Bigg\{U\subset \mathfrak{Map}_0^{C^k}\big(\mathcal{G},(\mathcal{H},\widehat{\mathcal{H}})\big)\ \Bigg| \  \begin{matrix} \forall P,\forall Q\in U\cap N_{\epsilon}^{C^k}(P), \exists \epsilon_Q>0  \\\text{ such that }\operatorname{Im}(\iota_{Q,P})\subset U\end{matrix}\Bigg\}.
\end{gather*}
The following proposition guarantees that this collection indeed forms a topology.
\begin{proposition}
    The collection ${ \mathcal{O}_0^{C^k}\big(\mathcal{G},(\mathcal{H},\widehat{\mathcal{H}})\big)}$ is a topology on ${\mathfrak{Map}_0^{C^k}\big(\mathcal{G},(\mathcal{H},\widehat{\mathcal{H}})\big)}$.
\end{proposition}

\begin{proof}
    We fix ${P\in \mathfrak{Map}_0^{C^k}\big(\mathcal{G},(\mathcal{H},\widehat{\mathcal{H}})\big)}$. For any ${Q\in N_{\epsilon}^{C^k}(P)}$ and a sufficiently small ${\epsilon_Q>0}$, we canonically identify ${N_{\epsilon_Q}^{C^k}(Q)}$ with a subset of ${N_{\epsilon}^{C^k}(P)}$. By the definition of the disjoint union topology, it suffices to prove that the collection 
    \begin{gather*}
        \big\{N_{\epsilon_Q}^{C^k}(Q) \ \big| \ Q\in N_{\epsilon}^{C^k}(P), \epsilon_Q>0 \text{ is sufficiently small}\big\}
    \end{gather*}
    forms a basis for the  induced topology on ${ N_{\epsilon}^{C^k}(P)}$. We fix ${P_1,P_2\in N_{\epsilon}^{C^k}(P)}$ and sufficiently small ${\epsilon_1,\epsilon_2>0}$.
    It suffices to prove that for any
    \begin{gather*}
        P_3\in N_{\epsilon_1}^{C^k}(P_1)\cap N_{\epsilon_2}^{C^k}(P_2),
    \end{gather*}
    there exists a sufficiently small ${\epsilon_3>0}$ such that
    \begin{gather*}
        N_{\epsilon_3}^{C^k}(P_3)\subset N_{\epsilon_1}^{C^k}(P_1)\cap N_{\epsilon_2}^{C^k}(P_2).
    \end{gather*}
    The bibundle $P_3$ is represented as
    \begin{gather*}
        P_3=P_1(f), \quad f\in T_{P_1}\mathfrak{Map}_0^{C^k}\big(\mathcal{G},(\mathcal{H},\widehat{\mathcal{H}})\big)^{<\epsilon_1}.
    \end{gather*}
    Let ${f^{\epsilon_1}\colon P_1^{\epsilon_1}\rightarrow T\widehat{H}_0^{\epsilon_1}}$ be the extension of $f$. Note that ${P_3\subset \widehat{P}_1}$ holds as underlying sets. However, because their right anchor maps as bibundles differ, we distinguish them by writing
    \begin{gather*}
        \widehat{a}_R^{(1)}\colon \widehat{P}_1\rightarrow \widehat{H}_0, \quad a_R^{(3)}\colon P_3\rightarrow H_0\subset \widehat{H}_0.
    \end{gather*}
    Then, the right anchor map ${a_R^{(3)}}$ satisfies
    \begin{gather*}
        a_R^{(3)}(x)=\exp_{\widehat{a}_R^{(1)}(x)}(f^{\epsilon_1}(x)) \quad (\forall x\in P_3).
    \end{gather*}
    For a sufficiently small ${\epsilon_3>0}$ and any ${g\in T_{P_3}\mathfrak{Map}_0^{C^k}\big(\mathcal{G},(\mathcal{H},\widehat{\mathcal{H}})\big)^{<\epsilon_3}}$, there exists ${h\in T_{P_1}\mathfrak{Map}_0^{C^k}\big(\mathcal{G},(\mathcal{H},\widehat{\mathcal{H}})\big)^{<\epsilon_1}}$ such that their extensions
    \begin{gather*}
        h^{\epsilon_1}\colon P_1^{\epsilon_1}\longrightarrow T\widehat{H}_0^{\epsilon_1}, \quad g^{\epsilon_3}\colon P_3^{\epsilon_3}\longrightarrow T\widehat{H}_0^{\epsilon_3}
    \end{gather*}
    satisfy ${P_3(g)=P_1(h)}$. This implies ${N_{\epsilon_3}^{C^k}(P_3)\subset N_{\epsilon_1}^{C^k}(P_1)}$. Similarly, by taking a smaller ${\epsilon_3>0}$ if necessary, ${N_{\epsilon_3}^{C^k}(P_3)\subset N_{\epsilon_2}^{C^k}(P_2)}$ also holds. Consequently, we obtain
    \begin{gather*}
        N_{\epsilon_3}^{C^k}(P_3)\subset N_{\epsilon_1}^{C^k}(P_1)\cap N_{\epsilon_2}^{C^k}(P_2),
    \end{gather*}
    which completes the proof.
\end{proof}

By this proposition, the topology ${\mathcal{O}_0^{C^k}\big(\mathcal{G},(\mathcal{H},\widehat{\mathcal{H}})\big)}$ is well-defined. Under this topology, the assignment ${f\mapsto P(f)}$ gives a local homeomorphism between the open $\epsilon$-ball ${T_P\mathfrak{Map}_0^{C^k}\big(\mathcal{G},(\mathcal{H},\widehat{\mathcal{H}})\big)^{<\epsilon}}$ and the neighborhood ${N_{\epsilon}^{C^k}(P)}$. This endows the mapping space ${\mathfrak{Map}_0^{C^k}\big(\mathcal{G},(\mathcal{H},\widehat{\mathcal{H}})\big)}$ with the structure of a Banach manifold whose tangent space at ${Q\in N_{\epsilon}^{C^k}(P)\subset \mathfrak{Map}_0^{C^k}\big(\mathcal{G},(\mathcal{H},\widehat{\mathcal{H}})\big)}$ is ${T_Q\mathfrak{Map}_0^{C^k}\big(\mathcal{G},(\mathcal{H},\widehat{\mathcal{H}})\big)}$.

Next, for an isomorphism between $C^k$ bibundles ${\Phi\colon P\rightarrow Q}$, we define a $\delta$-neighborhood of $\Phi$. For any ${f\in T_P\mathfrak{Map}_0^{C^k}\big(\mathcal{G},(\mathcal{H},\widehat{\mathcal{H}})\big)^{<\delta}}$, we define ${\Phi_*(f)\in T_Q\mathfrak{Map}_0^{C^k}\big(\mathcal{G},(\mathcal{H},\widehat{\mathcal{H}})\big)^{<\delta}}$ by
\begin{gather*}
    \Phi_*(f)(y)=f\big(\Phi^{-1}(y)\big) \quad (\forall y\in Q).
\end{gather*}
Note that ${\pi\big(\Phi_*(f)(y)\big)=a_{R,P}\big(\Phi^{-1}(y)\big)=a_{R,Q}(y)}$ holds because $\Phi$ is a bibundle isomorphism. Since $\Phi$ is $\mathcal{G}$-invariant and $\mathcal{H}$-equivariant, ${\Phi_*(f)}$ also satisfies the required invariance and equivariance condition.
Thus, this correspondence induces a bijection
\begin{gather*}
    \Phi_*\colon T_P\mathfrak{Map}_0^{C^k}\big(\mathcal{G},(\mathcal{H},\widehat{\mathcal{H}})\big)^{<\delta}\longrightarrow T_Q\mathfrak{Map}_0^{C^k}\big(\mathcal{G},(\mathcal{H},\widehat{\mathcal{H}})\big)^{<\delta} \\
    f\longmapsto \Phi_*(f).
\end{gather*}
By the construction of ${P(f)}$, $\Phi_*$ induces a diffeomorphism
\begin{gather*}
    \Phi_{*,0}\colon N_{\delta}^{C^k}(P)\longrightarrow N_{\delta}^{C^k}(Q)  \\  P(f)\longmapsto Q\big(\Phi_*(f)\big).
\end{gather*}
Moreover, $\Phi$ induces a natural isomorphism ${\Phi_{*,1}\colon P(f)\rightarrow Q\big(\Phi_*(f)\big)}$ between $C^k$ bibundles. Thus, we define the $\delta$-neighborhood of $\Phi$ by
\begin{gather*}
    N_{\delta}^{C^k}(\Phi)=\Big\{\Phi_{*,1}\colon P(f)\rightarrow Q\big(\Phi_*(f)\big) \ \Big| \ f\in T_P\mathfrak{Map}_0^{C^k}\big(\mathcal{G},(\mathcal{H},\widehat{\mathcal{H}})\big)^{<\delta}\Big\}.
\end{gather*}
We define the topology ${\mathcal{O}_1^{C^k}\big(\mathcal{G},(\mathcal{H},\widehat{\mathcal{H}})\big)}$ on ${\mathfrak{Map}_1^{C^k}\big(\mathcal{G},(\mathcal{H},\widehat{\mathcal{H}})\big)}$ by
\begin{gather*}
    \mathcal{O}_1^{C^k}\big(\mathcal{G},(\mathcal{H},\widehat{\mathcal{H}})\big)=\{U\subset \mathfrak{Map}_1^{C^k}\big(\mathcal{G},(\mathcal{H},\widehat{\mathcal{H}})\big) \mid \forall \Phi\in U, \exists \delta>0 \text{ s.t. }N_{\delta}^{C^k}(\Phi)\subset U \}.
\end{gather*}
Similarly to ${\mathcal{O}_0^{C^k}\big(\mathcal{G},(\mathcal{H},\widehat{\mathcal{H}})\big)}$, ${\mathcal{O}_1^{C^k}\big(\mathcal{G},(\mathcal{H},\widehat{\mathcal{H}})\big)}$ forms a topology on ${\mathfrak{Map}_1^{C^k}\big(\mathcal{G},(\mathcal{H},\widehat{\mathcal{H}})\big)}$. This follows from the following proposition.
\begin{proposition}
    The collection
    \begin{gather*}
        \{N_{\delta}^{C^k}(\Phi) \mid \Phi\in \mathfrak{Map}_1^{C^k}\big(\mathcal{G},(\mathcal{H},\widehat{\mathcal{H}})\big), \delta>0 \text{ is sufficiently small}\}
    \end{gather*}
    forms a basis for a topology ${\mathcal{O}_1^{C^k}\big(\mathcal{G},(\mathcal{H},\widehat{\mathcal{H}})\big)}$ on ${\mathfrak{Map}_1^{C^k}\big(\mathcal{G},(\mathcal{H},\widehat{\mathcal{H}})\big)}$.
\end{proposition}
The proof is completely analogous to that of Proposition 23, and is therefore omitted. By the construction of the topology ${\mathcal{O}_1^{C^k}\big(\mathcal{G},(\mathcal{H},\widehat{\mathcal{H}})\big)}$, the source map $s$ and the target map $t$ are local diffeomorphisms. Thus, we obtain the following proposition.
\begin{proposition}
    The source map and the target map
    \begin{gather*}
        s,t\colon \mathfrak{Map}_1^{C^k}\big(\mathcal{G},(\mathcal{H},\widehat{\mathcal{H}})\big)\longrightarrow \mathfrak{Map}_0^{C^k}\big(\mathcal{G},(\mathcal{H},\widehat{\mathcal{H}})\big)
    \end{gather*}
    are \'{e}tale.
\end{proposition}
Next, we establish the properness of ${s\times t}$.
\begin{proposition}
    The map
    \begin{gather*}
        s\times t\colon \mathfrak{Map}_1^{C^k}\big(\mathcal{G},(\mathcal{H},\widehat{\mathcal{H}})\big)\longrightarrow \mathfrak{Map}_0^{C^k}\big(\mathcal{G},(\mathcal{H},\widehat{\mathcal{H}})\big)\times \mathfrak{Map}_0^{C^k}\big(\mathcal{G},(\mathcal{H},\widehat{\mathcal{H}})\big)
    \end{gather*}
    is proper in $C^k$-topology.
\end{proposition}

\begin{proof}
    Let ${\{\Phi_k\}_{k\in \mathbb{N}}}$ be a sequence in ${\mathfrak{Map}_1^{C^k}\big(\mathcal{G},(\mathcal{H},\widehat{\mathcal{H}})\big)}$ such that
    \begin{gather*}
        s(\Phi_k)=P_k\stackrel{k\to \infty}{\longrightarrow} P, \quad t(\Phi_k)=Q_k\stackrel{k\to \infty}{\longrightarrow} Q.
    \end{gather*}
    Then we can choose ${f_k\in T_P\mathfrak{Map}_0^{C^k}\big(\mathcal{G},(\mathcal{H},\widehat{\mathcal{H}})\big)}$, ${l_k\in T_Q\mathfrak{Map}_0^{C^k}\big(\mathcal{G},(\mathcal{H},\widehat{\mathcal{H}})\big)}$ such that 
    \begin{gather*}
        P_k=P(f_k), \quad Q_k=Q(l_k).
    \end{gather*}
    We may assume that ${P_k\subset \widehat{P}}$ and ${Q_k\subset \widehat{Q}}$. We fix Riemannian metrics on $\widehat{P}$ and $\widehat{Q}$ induced by the left anchor maps ($=$ local diffeomorphisms) ${\widehat{a}_{L,\widehat{P}}\colon \widehat{P}\rightarrow G_0}$, ${\widehat{a}_{L,\widehat{Q}}\colon\widehat{Q}\rightarrow G_0}$. We denote the $\delta$-neighborhood of ${g\in G_0}$, ${p\in \widehat{P}}$, ${q\in \widehat{Q}}$ by ${N_{\delta}(g)}$, ${N_{\delta}(p)}$, ${N_{\delta}(q)}$, respectively. We define the $\delta$-neighborhood of ${[g]\in |\mathcal{G}|}$ by 
    \begin{gather*}
    N_{\delta}([g])=\bigcup_{g'\in \pi^{-1}([g])}\pi\big(N_{\delta}(g')\big)    
    \end{gather*}
    where ${\pi\colon G_0\rightarrow |\mathcal{G}|}$ is the natural projection. We fix a sufficiently small ${\delta>0}$ such that 
    \begin{gather*}
        \widehat{a}_{L,\widehat{P}}|_{N_{\delta}(p)}\colon N_{\delta}(p)\longrightarrow G_0, \quad \widehat{a}_{L,\widehat{Q}}|_{N_{\delta}(q)}\colon N_{\delta}(q)\longrightarrow G_0
    \end{gather*}
    are diffeomorphisms onto their images for any ${p\in P^{\epsilon}}$, ${q\in Q^{\epsilon}}$. First, for any ${g\in G_0}$, we show that we can choose a subsequence of ${\{\Phi_k\}}$ so that the subsequence converges to an isomorphism of principal right $\mathcal{H}$-bundles over ${g\in G_0}$. We extend each $\Phi_k$ to
    \begin{gather*}
        \Phi_k\colon P_k^{\epsilon}\longrightarrow Q_k^{\epsilon}.
    \end{gather*}
    Note that ${P_k^{\epsilon}\subset P^{2\epsilon}}$ and ${Q_k^{\epsilon}\subset Q^{2\epsilon}}$. For any ${x\in \widehat{a}_{L,\widehat{P}}^{-1}(g)\cap P_k^{\epsilon}}$, ${y_k=\Phi_k(x)\in Q^{2\epsilon}\cap \widehat{a}_{L,\widehat{Q}}^{-1}(g)}$ is an element of a finite set. Thus, after passing to a subsequence, we may assume that $y_k$ converges to some ${y\in Q^{2\epsilon}\cap \widehat{a}_{L,\widehat{Q}}^{-1}(g)}$. Then ${y\in Q}$, since
    \begin{gather*}
        \widehat{a}_{L,\widehat{Q}}(y)=a_{R,P}(x)\in H_0
    \end{gather*}
    holds. By the construction of $y_k$ and $y$, it follows that ${\Phi_k(x\cdot \gamma)}$ converges to ${y\cdot \gamma}$. Thus, we obtain an isomorphism of principal right $\mathcal{H}$-bundles over ${g\in G_0}$:
    \begin{gather*}
        \Phi_g\colon P|_{\{g\}}\longrightarrow Q|_{\{g\}}.
    \end{gather*}
    Then $\Phi_g$ induces an isomorphism of ${\mathcal{G}|_{[g]}}$-$\mathcal{H}$ bibundles:
    \begin{gather*}
        \Phi_{[g]}\colon P|_{[g]}\longrightarrow Q|_{[g]}.
    \end{gather*}
    Next, we show that ${\Phi_{[g]}}$ extends to the $\delta$-neighborhood of ${[g]\in |\mathcal{G}|}$. We fix ${g'\in N_{\delta}(g)\subset G_0}$ and ${x'\in N_{\delta}(x)\in P}$. By our choice of ${\delta}$, ${y_k'=\Phi_k(x')\in Q^{2\epsilon}\cap \widehat{a}_{L,\widehat{Q}}^{-1}(g')}$ is in the $\delta$-neighborhood of $y_k$ since the left anchor maps are local isometries on these neighborhoods. Thus, $y_k'$ also converges to ${y'\in N_{\delta}(y)}$. Therefore, ${\Phi_{[g]}}$ extends uniquely to the $\delta$-neighborhood of ${[g]}$:
    \begin{gather*}
        \Phi_{N_{\delta}([g])}\colon P|_{N_{\delta}([g])}\longrightarrow Q|_{N_{\delta}([g])}.
    \end{gather*}
    By the compactness of ${|\mathcal{G}|}$, choose finitely many such neighborhoods ${N_{\delta}([g_1]),\dots,N_{\delta}([g_m])}$ covering ${|\mathcal{G}|}$. Passing to subsequences successively, we may assume that $\Phi_k$ converges on each of these neighborhoods. Since the resulting local limits are all limits of the same subsequence, they agree on overlaps. Hence they glue to a global isomorphism of $\mathcal{G}$-$\mathcal{H}$ bibundles
    \begin{gather*}
        \Phi\colon P\longrightarrow Q.
    \end{gather*}
\end{proof}

The above proof can be applied to establish the Hausdorffness of the coarse moduli space.

\begin{proposition}
    The coarse moduli space ${|\mathfrak{Map}^{C^k}\big(\mathcal{G},(\mathcal{H},\widehat{\mathcal{H}})\big)|}$ is Hausdorff.
\end{proposition}

\begin{proof}
    We prove the proposition by contradiction. If ${|\mathfrak{Map}^{C^k}\big(\mathcal{G},(\mathcal{H},\widehat{\mathcal{H}})\big)|}$ is not Hausdorff, there exist ${P,Q\in \mathfrak{Map}_0^{C^k}\big(\mathcal{G},(\mathcal{H},\widehat{\mathcal{H}})\big)}$ with ${[P]\neq[Q]}$ in ${|\mathfrak{Map}^{C^k}\big(\mathcal{G},(\mathcal{H},\widehat{\mathcal{H}})\big)|}$, a decreasing sequence ${\epsilon_k\rightarrow 0}$, and tangent vectors
    \begin{gather*}
        f_k\in T_P\mathfrak{Map}_0^{C^k}\big(\mathcal{G},(\mathcal{H},\widehat{\mathcal{H}})\big)^{<\epsilon_k}, \quad l_k\in T_Q\mathfrak{Map}_0^{C^k}\big(\mathcal{G},(\mathcal{H},\widehat{\mathcal{H}})\big)^{<\epsilon_k},
    \end{gather*}
    together with isomorphisms of bibundles
    \begin{gather*}
        \Phi_k\colon P(f_k)\longrightarrow Q(l_k).
    \end{gather*}
    Since ${P(f_k)\to P}$ and ${Q(l_k)\to Q}$ as ${k\to \infty}$, by applying the proof of Proposition 26, we can choose a subsequence of ${\{\Phi_k\}}$ that converges to an isomorphism of bibundles
    \begin{gather*}
        \Phi\colon P\longrightarrow Q.
    \end{gather*}
    Consequently, ${[P]=[Q]}$ holds in ${|\mathfrak{Map}^{C^k}\big(\mathcal{G},(\mathcal{H},\widehat{\mathcal{H}})\big)|}$, which leads to a contradiction. Therefore, the coarse moduli space ${|\mathfrak{Map}^{C^k}\big(\mathcal{G},(\mathcal{H},\widehat{\mathcal{H}})\big)|}$ is Hausdorff.
\end{proof}

The following proposition implies that the Morita equivalence class of the Banach Lie groupoid ${\mathfrak{Map}^{C^k}\big(\mathcal{G},(\mathcal{H},\widehat{\mathcal{H}})\big)}$ does not depend on the choices of the representative orbifold groupoid $\mathcal{G}$ and the admissible pair ${(\mathcal{H},\widehat{\mathcal{H}})}$.

\begin{proposition}
    Let $\mathcal{G}'$ be an orbifold groupoid Morita equivalent to $\mathcal{G}$, and let ${\mathcal{H}'}$ be an admissible orbifold groupoid Morita equivalent to $\mathcal{H}$. For any admissible pair ${(\mathcal{H}',\widehat{\mathcal{H}}')}$, the two Banach Lie groupoids
    \begin{gather*}
        \mathfrak{Map}^{C^k}\big(\mathcal{G},(\mathcal{H},\widehat{\mathcal{H}})\big), \quad \mathfrak{Map}^{C^k}\big(\mathcal{G}',(\mathcal{H}',\widehat{\mathcal{H}}')\big)
    \end{gather*}
    are Morita equivalent.
\end{proposition}

\begin{proof}
    Let ${E\colon \mathcal{G}'\rightarrow \mathcal{G}}$ and ${F\colon \mathcal{H}\rightarrow \mathcal{H}'}$ be bibundles that are invertible in the category ${\mathbf{HS}}$. In other words, there exist bibundles ${E^{-1}\colon \mathcal{G}\rightarrow \mathcal{G}'}$, ${F^{-1}\colon \mathcal{H}'\rightarrow \mathcal{H}}$ such that
    \begin{gather*}
        E^{-1}\circ E\cong \langle \operatorname{Id}_{\mathcal{G}'}\rangle, \quad E\circ E^{-1}\cong \langle \operatorname{Id}_{\mathcal{G}}\rangle, \\
        F^{-1}\circ F\cong \langle \operatorname{Id_{\mathcal{H}}}\rangle, \quad F\circ F^{-1}\cong \langle \operatorname{Id}_{\mathcal{H}'}\rangle.
    \end{gather*}

    We construct a Banach bibundle from ${\mathfrak{Map}^{C^k}\big(\mathcal{G},(\mathcal{H},\widehat{\mathcal{H}})\big)}$ to ${\mathfrak{Map}^{C^k}\big(\mathcal{G}',(\mathcal{H}',\widehat{\mathcal{H}}')\big)}$ explicitly as follows:
    \begin{gather*}
        \Phi=\Bigg\{ (P,P',\phi) \ \Bigg| \ \begin{matrix}
            P\in\mathfrak{Map}_0^{C^k}\big(\mathcal{G},(\mathcal{H},\widehat{\mathcal{H}})\big), P'\in \mathfrak{Map}_0^{C^k}\big(\mathcal{G}',(\mathcal{H}',\widehat{\mathcal{H}}')\big) \\ \phi\colon P'\rightarrow (F\circ P)\circ E \text{ is an isomorphism of bibundles}
        \end{matrix}\Bigg\}  
    \end{gather*}
    with anchor maps defined by
    \begin{gather*}
        a_L\big((P,P',\phi)\big)=P,\quad a_R\big((P,P',\phi)\big)=P'.
    \end{gather*}
    Note that since $E$ and $F$ are smooth bibundles, the composition ${(F\circ P)\circ E}$ is a $C^k$ bibundle for each ${P\in \mathfrak{Map}_0^{C^k}\big(\mathcal{G},(\mathcal{H},\widehat{\mathcal{H}})\big)}$. For an isomorphism ${\gamma\colon P_1'\rightarrow P_2'\in \mathfrak{Map}_1^{C^k}\big(\mathcal{G}',(\mathcal{H}',\widehat{\mathcal{H}}')\big)}$, the right action is defined by
    \begin{gather*}
        (P,P_2',\phi)\cdot \gamma=(P,P_1',\phi\circ \gamma),
    \end{gather*}
    and for an isomorphism ${\tau\colon P_1\rightarrow P_2\in \mathfrak{Map}_1^{C^k}\big(\mathcal{G},(\mathcal{H},\widehat{\mathcal{H}})\big)}$, the left action is defined by
    \begin{gather*}
        \tau\cdot(P_1,P',\phi)=(P_2,P',\widetilde{\tau}\circ \phi)
    \end{gather*}
    where ${\widetilde{\tau}\colon (F\circ P_1)\circ E\rightarrow (F\circ P_2)\circ E }$ is the isomorphism induced by ${\tau}$. Note that the map
    \begin{gather*}
        \Big|\mathfrak{Map}^{C^k}\big(\mathcal{G},(\mathcal{H},\widehat{\mathcal{H}})\big)\Big|\longrightarrow \Big|\mathfrak{Map}^{C^k}\big(\mathcal{G}',(\mathcal{H}',\widehat{\mathcal{H}}')\big)\Big|  \\
        [P]\longmapsto [(F\circ P)\circ E]
    \end{gather*}
    is a bijection. Furthermore, by the construction of the local topologies via local deformation of bibundles, ${a_L}$ is a local diffeomorphism. These properties imply that ${a_L\colon \Phi\rightarrow \mathfrak{Map}_0^{C^k}\big(\mathcal{G},(\mathcal{H},\widehat{\mathcal{H}})\big)}$ is a principal right ${\mathfrak{Map}^{C^k}\big(\mathcal{G}',(\mathcal{H}',\widehat{\mathcal{H}}')\big)}$-bundle. Since the left action of ${\mathfrak{Map}^{C^k}\big(\mathcal{G},(\mathcal{H},\widehat{\mathcal{H}})\big)}$ commutes with the right action, $\Phi$ is a Banach bibundle from ${\mathfrak{Map}^{C^k}\big(\mathcal{G},(\mathcal{H},\widehat{\mathcal{H}})\big)}$ to ${\mathfrak{Map}^{C^k}\big(\mathcal{G}',(\mathcal{H}',\widehat{\mathcal{H}}')\big)}$. Similarly, we construct a bibundle $\Psi$ from ${\mathfrak{Map}^{C^k}\big(\mathcal{G}',(\mathcal{H}',\widehat{\mathcal{H}}')\big)}$ to ${\mathfrak{Map}^{C^k}\big(\mathcal{G},(\mathcal{H},\widehat{\mathcal{H}})\big)}$ by
    \begin{gather*}
        \Psi=\Bigg\{(P',P,\psi) \ \Bigg| \ \begin{matrix}
            P'\in \mathfrak{Map}_0^{C^k}\big(\mathcal{G}',(\mathcal{H}',\widehat{\mathcal{H}}')\big), P\in \mathfrak{Map}_0^{C^k}\big(\mathcal{G},(\mathcal{H},\widehat{\mathcal{H}})\big) \\
            \psi\colon P\rightarrow (F^{-1}\circ P')\circ E^{-1} \text{ is an isomorphism of bibundles}
        \end{matrix}\Bigg\}.
    \end{gather*}
    Then, it is straightforward to see that the composition ${\Psi\circ \Phi}$ is isomorphic to a bibundle $\Theta$ defined by
    \begin{gather*}
        \Theta=\Bigg\{(P_1,P_2,\theta) \ \Bigg| \ \begin{matrix}
            P_1,P_2\in \mathfrak{Map}_0^{C^k}\big(\mathcal{G},(\mathcal{H},\widehat{\mathcal{H}})\big) \\
            \theta\colon P_2\rightarrow \bigg(F^{-1}\circ \Big((F\circ P_1)\circ E\Big)\bigg)\circ E^{-1}\text{ is an isomorphism of bibundles}
        \end{matrix}\Bigg\}.
    \end{gather*}
    Because the composition of bibundles ${\bigg(F^{-1}\circ \Big((F\circ P_1)\circ E\Big)\bigg)\circ E^{-1}}$ is canonically isomorphic to the bibundle ${P_1}$, $\Theta$ is canonically isomorphic to the bibundle
    \begin{gather*}
        \big\langle \operatorname{Id}_{\mathfrak{Map}^{C^k}\big(\mathcal{G},(\mathcal{H},\widehat{\mathcal{H}})\big)}\big\rangle =\Bigg\{(P_1,P_2,\rho) \ \Bigg| \ \begin{matrix}
        P_1,P_2\in \mathfrak{Map}_0^{C^k}\big(\mathcal{G},(\mathcal{H},\widehat{\mathcal{H}})\big) \\ \rho\colon P_2\rightarrow P_1\in \mathfrak{Map}_1^{C^k}\big(\mathcal{G},(\mathcal{H},\widehat{\mathcal{H}})\big)
        \end{matrix}\Bigg\}.
    \end{gather*}
    Therefore, 
    \begin{gather*}
        \Psi\circ \Phi\cong \big\langle \operatorname{Id}_{\mathfrak{Map}^{C^k}\big(\mathcal{G},(\mathcal{H},\widehat{\mathcal{H}})\big)}\big\rangle
    \end{gather*}
    holds. By interchanging the roles of $\Phi$ and $\Psi$, we also obtain

\begin{gather*}
    \Phi\circ \Psi\cong \big\langle \operatorname{Id}_{\mathfrak{Map}^{C^k}\big(\mathcal{G}',(\mathcal{H}',\widehat{\mathcal{H}}')\big)}\big\rangle.
\end{gather*}
Thus, the two Banach Lie groupoids ${\mathfrak{Map}^{C^k}\big(\mathcal{G},(\mathcal{H},\widehat{\mathcal{H}})\big)}$ and ${\mathfrak{Map}^{C^k}\big(\mathcal{G}',(\mathcal{H}',\widehat{\mathcal{H}}')\big)}$ are Morita equivalent.
\end{proof}

By Proposition 28, we can denote the Banach groupoid ${\mathfrak{Map}^{C^k}\big(\mathcal{G},(\mathcal{H},\widehat{\mathcal{H}})\big)}$ by ${\mathfrak{Map}^{C^k}(\mathcal{G},\mathcal{H})}$. This Banach Lie groupoid admits local orbifold charts. In particular, it is locally isomorphic to the action groupoid of a finite group action.

\begin{proposition}
    For any $C^k$ bibundle ${Q\in \mathfrak{Map}_0^{C^k}(\mathcal{G},\mathcal{H})}$, there exist an open neighborhood $U$ of $Q$ and a finite group $\Gamma$ acting on $U$ such that
    \begin{gather*}
        \mathfrak{Map}^{C^k}(\mathcal{G},\mathcal{H})\Big|_{U}\cong \Gamma \ltimes U.
    \end{gather*}
\end{proposition}

\begin{proof}
    Let ${\Gamma\subset \mathfrak{Map}_1^{C^k}(\mathcal{G},\mathcal{H})}$ be the set of automorphisms of $Q$. Since the automorphism group $\Gamma$ of $Q$ is discrete and compact, it is a finite group
    \begin{gather*}
        \Gamma=\{\gamma_1,\gamma_2,\dots,\gamma_N\}.
    \end{gather*}
    Let $P'\in \mathfrak{Map}_0^{C^k}(\mathcal{G},\mathcal{H})$ be a $C^k$ bibundle such that
    \begin{gather*}
        Q\in N_{\epsilon}^{C^k}(P')\subset \mathfrak{Map}_0^{C^k}(\mathcal{G},\mathcal{H})=\coprod_{[P]\in \operatorname{Iso}^{C^k}(\mathcal{G},\mathcal{H})}N_{\epsilon}^{C^k}(P).
    \end{gather*}
    For a sufficiently small ${\epsilon_Q>0}$, we define the open neighborhood $U$ and the subspace $\widetilde{\Gamma}$ by
    \begin{gather*}
        U=N_{\epsilon_Q}^{C^k}(Q)\subset N_{\epsilon}^{C^k}(P'), \quad \widetilde{\Gamma}=\bigcup_{1\le i\le N}N_{\epsilon_Q}^{C^k}(\gamma_i).
    \end{gather*}
    Similarly to the proof of Proposition 26, we can show that for sufficiently small ${\epsilon_Q>0}$,
    \begin{gather*}
        \big\{\gamma\in \mathfrak{Map}_1^{C^k}(\mathcal{G},\mathcal{H}) \ \big| \  \big(s(\gamma),t(\gamma)\big)\in U\times U\big\}=\widetilde{\Gamma}
    \end{gather*}
    holds. Therefore, 
    \begin{gather*}
        \mathfrak{Map}^{C^k}(\mathcal{G},\mathcal{H})\Big|_{U}\cong \Gamma\ltimes U
    \end{gather*}
    holds.
\end{proof}

\begin{remark}
    The same arguments apply to the $C^{\infty}$ and $W^{k,p}$ settings, with the $W^{k,p}$ case relying on the Sobolev embedding ${W^{k,p}\hookrightarrow C^{0}}$. The relevant constructions and arguments are unchanged at the level of the local Banach and Fr\'{e}chet models. Thus, we obtain a Fr\'{e}chet orbifold groupoid ${\mathfrak{Map}^{C^{\infty}}(\mathcal{G},\mathcal{H})}$ and a Banach orbifold groupoid ${\mathfrak{Map}^{W^{k,p}}(\mathcal{G},\mathcal{H})}$.
\end{remark}
Summarizing the results established above, we arrive at the main theorem of this paper.

\begin{theorem}[Main Theorem]
    Let $\mathcal{G}$, $\mathcal{H}$ be orbifold groupoids with compact coarse moduli spaces ${|\mathcal{G}|}$, ${|\mathcal{H}|}$.
    \begin{enumerate}
        \item 
        For any ${k\in\mathbb{Z}_{\ge 0}}$, there exists a proper \'{e}tale Banach Lie groupoid ${\mathfrak{Map}^{C^k}(\mathcal{G},\mathcal{H})}$ that is locally isomorphic to the action groupoid ${\Gamma\ltimes U}$ of a finite group action. Moreover, the Morita equivalence class of ${\mathfrak{Map}^{C^k}(\mathcal{G},\mathcal{H})}$ does not depend on the choices of the representatives $\mathcal{G}$ and ${\mathcal{H}}$, and its coarse moduli space ${\big|\mathfrak{Map}^{C^k}(\mathcal{G},\mathcal{H})\big|}$ is Hausdorff and in bijection with the set of isomorphism classes of $C^k$ bibundles from $\mathcal{G}$ to $\mathcal{H}$.
        \item There exists a proper \'{e}tale Fr\'{e}chet Lie groupoid ${\mathfrak{Map}^{C^{\infty}}(\mathcal{G},\mathcal{H})}$ that is locally isomorphic to the action groupoid ${\Gamma\ltimes U}$ of a finite group action. Moreover, the Morita equivalence class of ${\mathfrak{Map}^{C^{\infty}}(\mathcal{G},\mathcal{H})}$ does not depend on the choices of the representatives $\mathcal{G}$ and $\mathcal{H}$, and the coarse moduli space ${\big|\mathfrak{Map}^{C^{\infty}}(\mathcal{G},\mathcal{H})\big|}$ is Hausdorff and in bijection with the set of isomorphism classes of $C^{\infty}$ bibundles from $\mathcal{G}$ to $\mathcal{H}$.
        \item For any ${k,p}$ with ${kp>\mathrm{dim}\mathcal{G}}$, there exists a proper \'{e}tale Banach Lie groupoid ${\mathfrak{Map}^{W^{k,p}}(\mathcal{G},\mathcal{H})}$ that is locally isomorphic to the action groupoid ${\Gamma\ltimes U}$ of a finite group action. Moreover, the Morita equivalence class of ${\mathfrak{Map}^{W^{k,p}}(\mathcal{G},\mathcal{H})}$ does not depend on the choices of representatives $\mathcal{G}$ and ${\mathcal{H}}$, and the coarse moduli space ${\big|\mathfrak{Map}^{W^{k,p}}(\mathcal{G},\mathcal{H})\big|}$ is Hausdorff and in bijection with the set of isomorphism classes of ${W^{k,p}}$ bibundles from $\mathcal{G}$ to ${\mathcal{H}}$.
    \end{enumerate}
\end{theorem}

\end{document}